\documentclass[english,11pt, twoside]{article}
\usepackage[latin1]{inputenc}
\usepackage[T1]{fontenc}
\usepackage{lmodern}
\usepackage[a4paper]{geometry}
\usepackage{mathtools}
\usepackage{amssymb}
\usepackage{amsthm}
\usepackage{mathrsfs}
\usepackage{xcolor}
\usepackage{latexsym}
\usepackage{stmaryrd}
\usepackage{soul}
\usepackage{booktabs}
\usepackage{multirow}
\usepackage{placeins}
\usepackage{pgfmath}
\usepackage{xargs}
\usepackage{dsfont} 
\usepackage{appendix}
\usepackage{shuffle}
\usepackage{ifthen}
\usepackage{supertabular}
\usepackage{csquotes}
\usepackage{babel}
\usepackage[autolanguage,np]{numprint}
\usepackage[breaklinks]{hyperref}
\usepackage{enumitem}
\usepackage{listings} 
\newtheorem{defi}{\indent Definition}
\newtheorem{lemma}[defi]{\indent Lemma}
\newtheorem{cor}[defi]{\indent Corollary}
\newtheorem{theo}[defi]{\indent Theorem}
\newtheorem{prop}[defi]{\indent Proposition}
\newtheorem{conj}[defi]{\indent Conjecture}

\newcommand{\N}{\mathbb{N}} 
\newcommand{\C}{\mathbb{C}} 
\newcommand{\R}{\mathbb{R}} 
\newcommand{\K}{\mathbb{K}} 

\newcommand{\ot}{\otimes} 
\newcommand{\FQSym}{\mathbf{FQSym}} 
\newcommand{\Rq}[1]{\ifthenelse{\equal{#1}{}}{\subparagraph{Remarks.}}{\subparagraph{Remark.}}} 
\newcommand{\Ex}[1]{\ifthenelse{\equal{#1}{}}{\subparagraph{Examples.}}{\subparagraph{Example.}}} 
\DeclareMathOperator{\length}{length}

\DeclareMathOperator{\Span}{Span}

\begin{document}
\title{Weak stuffle algebras}
\date{}
\author{Cécile Mammez\\
	\small{\it Univ. Lille, UMR 8524 - Laboratoire Paul Painlevé,  F-59000 Lille, France}\\
	\small{\it CNRS, UMR 8524,  F-59000 Lille, France} \\ 
	\small{e-mail : cecile.mammez@univ-lille.fr}
}

\maketitle

\textbf{Abstract.} Motivated by $q$-shuffle products determined by Singer from $q$-analogues of multiple zeta values, we build in this article a generalisation of the shuffle and stuffle products in terms of weak shuffle and stuffle products. Then, we characterise weak shuffle products and give as examples the case of an alphabet of cardinality two or three. We focus on a comparison between algebraic structures respected in the classical case and in the weak case. As in the classical case, each weak shuffle product can be equipped with a dendriform structure. However, they have another behaviour towards the quadri-algebra and the Hopf algebra structure. We give some relations satisfied by weak stuffle products.  \\


\textbf{Résumé.} A partir de $q$-analogues aux fonctions zêta multiples, Singer détermine  des $q$-battages. Ceci motive, dans cet article, la construction d'une généralisation des produits de battage et de battage contractant en produits de battage faibles et produits de battage contractant faibles. Nous caractérisons  ensuite les battages faibles et donnons comme exemple le cas d'un alphabet à deux ou trois lettres. Nous comparons les structures algébriques respectées dans le cas classique et dans le cas faible. Comme dans le cas classique, tout battage faible peut être muni d'une structure d'algèbre dendriforme. En revanche, ils se comportent différemment face à la tructure de quadri-algèbre et d'algèbre de Hopf.  Nous donnons des relations vérifiées par les battages contractants faibles.  \\

\textbf{Mots-clés.} Algèbres de battage, algèbres de battage contractant, algèbres dendriformes, quadri-algèbres, algèbres de Hopf. \\


\allowdisplaybreaks
\section*{Introduction}  
The notion of shuffle and stuffle algebras is widely used in several fields of mathematics. Indeed, they participate in the study of Rota-Baxter algebras with the notion of mixable shuffle algebras \cite{Guo2000,Ebrahimi-Fard2006,Jian2017}, in the study of Yang-Baxter algebras \cite{Jian2010}, in the study of quasi-symmetric functions and words algebras \cite{Gessel1984,Malvenuto1994,Malvenuto1995,Duchamp2002,Duchamp2011,Foissy2016,Vargas2014,Mammez}, in the study of multiple zeta values \cite{Zagier1994,Hoffman1997,Hoffman2000,Hoffman2003a,Hoffman2017,Hoffman2018,Singer2016,Ebrahimi-Fard2016,Ebrahimi-Fard2016a} \dots 

The classical stuffle product comes from the product of classical multiple zeta values and is defined by the relation 
\[au\Box bv=a(u\Box bv)+b(au\Box v)+ (a\diamond b)(u\Box v) \]
where $a$ and $b$ are letters, $u$ and $v$ are words and $\diamond$ is an associative and commutative product which is equal to $0$ in the case of the classical shuffle product. Thus, the shuffle part of the relation is symmetric and does not depend on letters of any words in the product. 
In his work, Singer focuses on $q$-shuffle products coming from $q$-analogues of multiples zeta values. This case enables the existance of some letters $p$ and $y$ satisfying a relation in the form of 
\[yu\Box pv=pv\Box yu=y(u\Box pv) \] for any words $u$ and $v$. 
This new $q$-shuffle relation is not symmetric and depends on the beginning of each word in the product. This leads to focus on new generalisations of shuffle and stuffle products \cite{Singer2016a,Ebrahimi-Fard2016a,Ebrahimi-Fard2016}. 

In this article, we present a new generalisation of shuffle and stuffle algebras, we study their algebraic structures and compare them to the classical case. The article is organised as follows.
\begin{itemize}
	\item In Section $1$, we recall the classical notion of shuffle and stuffle product thanks to the multiple zeta values as well as the calculation by Singer of $q$-shuffle associated to the Schlesinger-Zudilin  model and the Bradley-Zhao model.
	\item In Section $2$, we define a generalisation of the classical shuffle product and the classical stuffle product called weak shuffle products and weak stuffle products and prove a characterisation of weak shuffle products. We detail the case of an alphabet of cardinality $2$ or $3$.
	\item In Section $3$, we focus on algebraic structures respected by the classical shuffle product and we determine if the weak shuffle products respect them too. Thus we prove that weak shuffle products are dendriform but there are obstacles to the quadri-algebra structure.
	\item In Section $4$, we express some relations satistied by weak stuffle products and we express the $q$-shuffle products given by Singer in terms of weak stuffle product. Besides, in the case of an infinite, countable and totally ordered alphabet $\{x_1,\dots,x_n,\dots\}$, we prove that, if the contracting part in the weak stuffle products is expressed as $f_3(x_i\ot x_j)\in\K^{*}x_{i+j}$, then the shuffle part is the null product or the classical shuffle product.  We give some informations more about weak stuffle products in the case of an alphabet of cardinality $2$ or $3$. 
	\item In Section $5$, we prove that a weak stuffle product is compatible with the deconcatenation coproduct if and only if the underlying weak shuffle product is the classical shuffle product and the contracting part is associative and commutative.
	\item Computation programs used to prove Lemma \ref{lemmaprogramme} are detailed in Section $6$.
\end{itemize}

\textbf{Acknowledgment:} I would like to thank the anonymous referees for their useful comments and suggestions. I would like to thank all peaple who supported me for this work. This work was funding by the Laboratoire Paul Painlevé at Université de Lille, the Fédération de Recherche Mathématique des Hauts-de-France, the ANR Alcohol project ANR-19-CE40-0006 and  the Labex CEMPI ANR-11-LABX-0007-01.
 
\section{Reminders}
\subsection{Classical shuffle and stuffle algebras}
We recall here the definition of the stuffle product in the context of the multiple zeta values.

\begin{defi}
	Let $s$ be an integer and let $(k_1,\dots,k_s)$ be an $s$-tuple in $\N_{\geq 2}\times\N^{s-1}$. The multiple zeta value associated to $(k_1,\dots,k_s)$
	is \[\displaystyle\zeta(k_1,\dots,k_s)=\sum_{\substack{(m_1,\dots,m_s)\in\N\\ m_1>\dots>m_s>0}}\frac{1}{m_1^{k_1}\dots m_s^{k_s}}.\]
\end{defi}
On multiple zeta values, we consider the product of functions taking values in $\C$. For instance,
\begin{align*}
	\zeta(n)\zeta(m)=&\zeta(m,n)+\zeta(n,m)+\zeta(m+n),\\
	\zeta(n,p)\zeta(m)=&\zeta(m,n,p)+\zeta(n,m,p)+\zeta(n,p,m)+\zeta(n+m,p)+\zeta(n,p+m).
\end{align*}
Then, it leads to the following algebraic definition and following theorem \cite{Hoffman1997}.
\begin{theo}
	Let $X=\{x_1,\dots, x_n,\dots\}$ be a countable alphabet. Let $\K\langle X\rangle$ be the algebra of words on the alphabet $X$.  We define the product $\star$, called the stuffle product, by:
	\begin{align*}
	 	u\star 1=&1\star u=1,\\
		u\star 0=&0\star u=0,\\
		x_iu\star x_jv=&x_i(u\star x_jv)+x_j(x_iu\star v)+x_{i+j}(u\star v)
	\end{align*}
for any letters $x_i$ and $x_j$ and any words $u$ and $v$.

Then 
\begin{align*}
x_iux_k\star x_jvx_l=&x_i(ux_k\star x_jvx_l)+x_j(x_iux_k\star vx_l)+x_{i+j}(ux_k\star vx_l)\\
=&(x_iu\star x_jvx_l)x_k+(x_iux_k\star x_jv)x_l+(x_{i}u\star x_jv)x_{k+l}
\end{align*}
and $(\K\langle X\rangle,\star)$ is an associative and commutative algebra.
\end{theo}

It is possible to define another algebra:
\begin{theo}
	Let $X=\{x_1,\dots, x_n,\dots\}$ be a countable alphabet. Let $\K\langle X\rangle$ be the algebra of words on the alphabet $X$.  We define the product $\shuffle$, called the shuffle product, by:
	\begin{align*}
	u\shuffle 1=&1\shuffle u=1,\\
	u\shuffle 0=&0\shuffle u=0,\\
	x_iu\shuffle x_jv=&x_i(u\shuffle x_jv)+x_j(x_iu\shuffle v)
	\end{align*}
	for any letters $x_i$ and $x_j$ and any words $u$ and $v$.
	
	Then 
	\begin{align*}
	x_iux_k\shuffle x_jvx_l=&x_i(ux_k\shuffle x_jvx_l)+x_j(x_iux_k\shuffle vx_l)\\
	=&(x_iu\shuffle x_jvx_l)x_k+(x_iux_k\shuffle x_jv)x_l
	\end{align*}
	and $(\K\langle X\rangle,\shuffle)$ is an associative and commutative algebra.
\end{theo}

\begin{theo}
	Let $X=\{x_1,\dots, x_n,\dots\}$ be a countable alphabet. The algebras $(\K\langle X\rangle,\star)$ and $(\K\langle X\rangle,\shuffle)$ are isomorphic.
\end{theo}

\begin{proof}
	This theorem was proved by Hoffman \cite[Theorem 2.5]{Hoffman2000} by describing an explicit isomorphism $\exp$. Another construction of $\exp$ leading to the proof of this theorem is given in \cite[Proposition 41]{Mammez}.
\end{proof}

\subsection{$q$-shuffle products for the Schlesinger-Zudilin  model and the Bradley-Zhao model.}
Let $q$ be real number such that $0<q<1$. A $q$-analogue of a positive integer $m$ is defined by 
\[[m]_q=\frac{1	-q^m}{1-q}=1+q+\dots+q^{m-1}.\]
The Schlesinger-Zudilin model \cite{Schlesinger2001,Zudilin2003} is defined as the following $q$-sum:
\begin{align*}
	\zeta_{q}^{SZ}(k_1,\dots,k_n)=&(1-q)^{-(k_1+\dots +k_n)}\sum_{\substack{(m_1,\dots,m_s)\in\N\\ m_1>\dots>m_s>0}}\frac{q^{m_1k_1+\dots+m_nk_n}}{[m_1]_q^{k_1}\dots [m_n]_q^{k_n}}\\=&
	\sum_{\substack{(m_1,\dots,m_s)\in\N\\ m_1>\dots>m_s>0}}\frac{q^{m_1k_1+\dots+m_nk_n}}{(1-q^{m_1})^{k_1}\dots (1-q^{m_n})^{k_n}}
\end{align*}
for any $(k_1,\dots,k_n)\in(\N^*)^n$.

The Bradley-Zhao model \cite{Bradley2005,Zhao2007} is defined as the following $q$-sum:
\begin{align*}
\zeta_{q}^{BZ}(k_1,\dots,k_n)=&(1-q)^{-(k_1+\dots +k_n)}\sum_{\substack{(m_1,\dots,m_s)\in\N\\ m_1>\dots>m_s>0}}\frac{q^{m_1(k_1-1)+\dots+m_n(k_n-1)}}{[m_1]_q^{k_1}\dots [m_n]_q^{k_n}}\\=&
\sum_{\substack{(m_1,\dots,m_s)\in\N\\ m_1>\dots>m_s>0}}\frac{q^{m_1(k_1-1)+\dots+m_n(k_n-1)}}{(1-q^{m_1})^{k_1}\dots (1-q^{m_n})^{k_n}}
\end{align*}
for any $(k_1,\dots,k_n)\in\N^n$ with $k_1\geq 2$.

From those two models, Singer defined two $q$-shuffle products corresponding to the algebraic version of the Schlesinger-Zudilin model and the Bradley-Zhao model and proved the following two theorems in \cite{Singer2015,Singer2016,Singer2016a}:

\begin{theo}[Singer]
	Let $X=\{y,p\}$ be an alphabet. The $q$-shuffle product associated to the Schlesinger-Zudilin model is given by:
	for any words $u$ and $v$,
	\begin{enumerate}
		\item $1 \shuffle_{SZ} u=u\shuffle_{SZ}1=u$,
		\item $yu\shuffle_{SZ}v=v\shuffle_{SZ}yu=y(u\shuffle_{SZ}v)$,
		\item $pu\shuffle_{SZ}pv=p(u\shuffle_{SZ}pv)+p(pu\shuffle_{SZ}v)+p(u\shuffle_{SZ}v)$.
	\end{enumerate}
	Besides, it is an associative and commutative product.
\end{theo}

\begin{theo}[Singer]
	Let $X=\{y,p,\overline{p}\}$ be an alphabet. The $q$-shuffle product associated to the Bradley-Zhao model is given by:
	for any words $u$ and $v$,
	\begin{enumerate}
		\item $1 \shuffle_{BZ} u=u\shuffle_{BZ}1=u$,
		\item $yu\shuffle_{BZ}v=v\shuffle_{BZ}yu=y(u\shuffle_{BZ}v)$,
		\item $au\shuffle_{BZ}bv=a(u\shuffle_{BZ}bv)+b(au\shuffle_{BZ}v)+[a,b]a(u\shuffle_{BZ}v)$ where \[a,b\in\{p,\overline{p}\},~ [p,p]=-[\overline{p},\overline{p}]=1 \text{ and } [p,\overline{p}]=[\overline{p},p]=0.\]
	\end{enumerate}
	Besides, it is an associative and commutative product.
\end{theo}

\section{Definition and characterisation of weak shuffle products}
The aim of this section is to define a generalisation of the classical shuffle product, the classical stuffle product, and the two $q$-shuffle products given by the Schlesinger-Zudilin model and the Bradley-Zhao model. We give and prove a characterisation of weak shuffle products too. Then we explicit the case of an alphabet of cardinality $2$ or $3$.
\subsection{Characterisation}
\begin{defi}
	An alphabet is a non-empty finite or countable set $X$.
\end{defi}

\begin{defi}
	Let $X$ be an alphabet. We denote by $X^*$ the set of words on the alphabet $X$ and by  $\K\langle X\rangle$ the tensor algebra generated by $X$ (\emph{i.e.} the algebra of words on $X$). The space $\K\langle X\rangle$ is graded by the length of words.
\end{defi}

\begin{defi}
	Let $X$ be an alphabet. A weak stuffle product on $\K\langle X\rangle$ is an associative and commutative product $\Box$ such that for any $(a,b)\in(X)^2$ and any $(u,v)\in(X^*)^2$ 
	\begin{align*}
	u\Box 1=&1\Box u=u,\\
	u\Box 0=&0\Box u=0,\\
	au\Box bv=& f_1(a\ot b)a(u\Box bv)+f_2(a\ot b)b(au\Box v)+f_3(a\ot b)(u\Box v)
	\end{align*}
	
	where 
	\begin{enumerate}
		\item $f_1$ and $f_2$ are linear maps from $\K.X\ot \K.X$ to $\K$,
		\item $f_3=kg$ is a linear map from $\K.X\ot \K.X$ to $\K.X$ such that $k(a\ot b)\in\K$ and $g(a\ot b)\in X$ for any $(a,b)\in X^2$,
		\item If $f_3\equiv 0$ then the product $\Box$ is called a weak shuffle product.
	\end{enumerate}
\end{defi}
\Ex{} Let $X=\{x_1,\dots,x_n,\dots\}$ be an infinite alphabet.
\begin{enumerate} 
	\item The classical shuffle product on $\K\langle X\rangle$ is a weak stuffle product where $f_1(a\otimes b )=1$ and  $f_2(a\otimes b)=1$ for any  $(a,b)\in X^2$, and $f_3\equiv0$.
	\item The classical stuffle product on $\K\langle X\rangle$ is a weak stuffle product where $f_1(a\otimes b )=1$ and  $f_2(a\otimes b)=1$ for any  $(a,b)\in X^2$, and $f_3(x_i\ot x_j)=x_{i+j}$ for any $(i,j)\in(\N^*)^2$.
	\item The stuffle product on $\K\langle X\rangle$ given by Hoffman and Ihara in \cite{Hoffman2017} is a weak stuffle product where $f_1(a\otimes b )=1$ and  $f_2(a\otimes b)=1$ for any  $(a,b)\in X^2$, and  $f_3(x_i\ot x_j)=-x_{i+j}$ for any $(i,j)\in(\N^*)^2$.
\end{enumerate}

\begin{theo} \label{carctwshuffle} Let $\Box$ be a product on $\K\langle X\rangle$. The map $\Box$ is a weak shuffle product if and only if, for any distinct letters $a$, $b$, and $c$ in $X$: 
\begin{enumerate}
		\item $f_1(a\ot b)=f_2(b\ot a)$. \label{relun}
		\item	\label{reltrois}\begin{enumerate}
			\item \label{reltroisun} either $f_1(a\ot a)=f_2(a\ot a)=\alpha$ with $\alpha\in\{0,1\}$ and
			\begin{enumerate}
				\item $f_1(a\ot b)f_1(b\ot a)[f_1(a\ot a)-1]=0$,
				\item $f_1(a\ot a)f_1(a\ot b)[f_1(a\ot b)-1]=0$,
				\item $f_1(a\ot a)f_1(b\ot a)[f_1(b\ot a)-1]=0$.
			\end{enumerate}
			\item \label{reltroisdeux} or $f_1(a\ot a)=\alpha$, $f_2(a\ot a)=1-\alpha$ with $\alpha\in\R$ and
			\begin{enumerate}
				\item $f_1(a\ot b)=1$,
				\item $f_1(b\ot a)=0$.
			\end{enumerate}
		\end{enumerate}
		\item $f_1(a\ot b)f_1(b\ot c)[f_1(a\ot c)-1]=0$. \label{relquatre}
		\item $f_3\equiv0$. \label{relcinq}
	\end{enumerate}	
\end{theo}

\Rq{1} It is sometimes usefull to use in calculations the following statement  induced by the item \ref{reltroisdeux} of the Theorem \ref{carctwshuffle}:

"If $f_1(a\ot b)=0$ or $f_1(b\ot a)\neq0$ then  $f_1(a\ot a)=f_2(a\ot a)=\alpha$ with $\alpha\in\{0,1\}$." \label{reldeux}

\begin{proof}	
	 Let us prove first the direct implication. Let us assume $\Box$ is a weak shuffle product. Let $a$, $b$, and $c$ be three distinct letters.  Then, by direct calculations, 
	\begin{enumerate}[label=(\Alph*)]
		\item \label{relA} $a\Box b=b\Box a$ gives relation $f_1(a\ot b)=f_2(b\ot a)$.
		\item \label{relB} $a\Box aa=aa\Box a$ gives $f_1(a\ot a)=f_2(a\ot a)$ or $f_1(a\ot a)=1-f_2(a\ot a)$.
		\item \label{relC} $a\Box ab=ab\Box a$ gives, if $f_1(a\ot b)=0$ or $f_1(b\ot a)\neq0$, that $f_1(a\ot a)=f_2(a\ot a)$. Thus, if $f_1(a\ot a)=1-f_2(a\ot a)$ and  $f_1(a\ot a)\neq \frac{1}{2}$ then $f_1(a\ot b)\neq0$ and $f_1(b\ot a)=0$.  The relation $a\Box ab=ab\Box a$ implies $f_1(a\ot b)=1$.
		\item \label{relD} $(a\Box a)\Box b=a\Box (a\Box b)=(a\Box b)\Box a$ with $f_1(a\ot a)=f_2(a\ot a)$  give
		\begin{enumerate}
			\item $f_1(a\ot b)f_1(b\ot a)[f_1(a\ot a)-1]=0$,
			\item $f_1(a\ot a)f_1(a\ot b)[f_1(a\ot b)-1]=0$,
			\item $f_1(a\ot a)f_1(b\ot a)[f_1(b\ot a)-1]=0$.
		\end{enumerate}
		\item \label{relE} $(a\Box b)\Box c=a\Box(b\Box c)$ gives $f_1(a\ot b)f_1(b\ot c)[f_1(a\ot c)-1]=0$.
		\item \label{relF} $(a\Box a)\Box ab=a\Box (a\Box ab)$ implies that if $f_1(a\ot a)=1-f_2(a\ot a)=\frac{1}{2}$ then $f_1(a\ot b)=1$ and $f_1(b\ot a)=0$.
		\item \label{relG} $(a\Box a)\Box aa=a\Box (a\Box aa)$ and $(a\Box a)\Box aaa=a\Box (a\Box aaa)$ implies that if $f_1(a\ot a)=f_2(a\ot a)=\alpha$ then $\alpha\in\{0,1,\frac{1}{2}\}$. 
		\item \label{relH} Cases $ba\Box a=a\Box ba$, $aa\Box b=b\Box aa$, $ab\Box c=c\Box ab$ and $(a\Box a)\Box a=a\Box (a\Box a)$ do not give any further relations.
	\end{enumerate}
	As a consequence, in the theorem \ref{carctwshuffle}, 
	\begin{itemize}
		\item the item \ref{relun} is proved by the item \ref{relA},
		\item the item \ref{reltroisun} is proved by the items \ref{relB}, \ref{relD}, \ref{relF} and \ref{relG},
		\item the item \ref{reltroisdeux} is proved by the items \ref{relB}, \ref{relC} and \ref{relF},
		\item the item \ref{relquatre} is proved by the item \ref{relE},
		\item the item \ref{relcinq} is satisfied by the definition of a weak shuffle product.
	\end{itemize}
	
	Conversly, if $\Box$ satisfies all relations given in Theorem \ref{carctwshuffle} then for any couple $(u,v)$ and any triple $(w_1,w_2,w_3)$ of words such that $\length(u)+\length(v)\leq 3$ and $\length(w_1)+\length(w_2)+\length(w_3)\leq 3$ one has: $u\Box v=v\Box u$ and $(w_1\Box w_2)\Box w_3=w_1\Box (w_2\Box w_3)$. 
	
	We assume now there exists an integer $n\geq3$ such that $u\Box v=v\Box u$ and $(w_1\Box w_2)\Box w_3=w_1\Box (w_2\Box w_3)$ for any words $u$, $v$, $w_1$, $w_2$ with $\length(u)+\length(v)\leq n$ and $\length(w_1)+\length(w_2)+\length(w_3)\leq n$.
	
	Let now $u$ and $v$ be two words such that $\length(u)+\length(v)=n+1$. Then there exist two letters $a$ and $b$ and two words $w_1$ and $w_2$ (not necessarily non-empty) such that $u=aw_1$ and $v=bw_2$. Then, by induction, we get:
	\begin{description}
		\item[case $a\neq b$.] \begin{align*}
			u\Box v=&f_1(a\ot b)a(w_1\Box bw_2)+f_1(b\ot a)b(aw_1\Box w_2)\\
			=&f_1(a\ot b)a(bw_2\Box w_1)+f_1(b\ot a)b(w_2\Box aw_1)=v\Box u.
		\end{align*}
		\item[case $a=b$ and $f_1(a\ot a)=f_2(a\ot a)$.] \begin{align*}
			u\Box v=&f_1(a\ot a)a(w_1\Box aw_2)+f_1(a\ot a)a(aw_1\Box w_2)\\
			=&f_1(a\ot a)a(aw_2\Box w_1)+f_1(a\ot a)a(w_2\Box aw_1)=v\Box u.
		\end{align*} 
		\item[case $a=b$ and $f_2(a\ot a)=1-f_1(a\ot a)$.] There exist two words $w_3$ and $w_4$, not necessarily non-empty, not starting by $a$ and two positive integers $k$ and $l$ such that $w_1=\underbrace{a\dots a}_{k \text{ times}}w_3$ and $w_2=\underbrace{a\dots a}_{l \text{ times}}w_4$. 
		First of all, by induction, \[\underbrace{a\dots a}_{k \text{ times}}\Box\underbrace{a\dots a}_{l \text{ times}}=\underbrace{a\dots a}_{k+l \text{ times}}.\]
		Besides, relations satisfied by $\Box$ enjoin $f_1(a\ot c)=1$ and $f_2(c\ot a)=0$ for any letter $c\neq a$. So,
		\[u\Box v=(\underbrace{a\dots a}_{k \text{ times}}\Box \underbrace{a\dots a}_{l \text{ times}})(w_3\Box w_4)=(\underbrace{a\dots a}_{l \text{ times}}\Box \underbrace{a\dots a}_{k \text{ times}})(w_4\Box w_3)=v\Box u.\]
	\end{description}
	As a consequence, $\Box$ is a commutative product.
	
	Let now $w_1$, $w_2$ and $w_3$ be three words such that $\length(w_1)+\length(w_2)+\length(w_3)=n+1$. Then there exist three letters $a$, $b$ and $c$ and three words $w_4$, $w_5$ and $w_6$ (not necessarily non-empty) such that $w_1=aw_4$, $w_2=bw_5$ and $w_3=cw_6$. Then, by induction, we get:
		\begin{description}
		\item[case $a$, $b$ and $c$ distinct.] \begin{align*}
		(w_1\Box w_2)\Box w_3=&f_1(a\ot b)f_1(a\ot c)a[(w_4\Box bw_5)\Box cw_6]+f_1(a\ot b)f_1(c\ot a)c[a(w_4\Box bw_5)\Box w_6]\\
		+&f_1(b\ot a)f_1(b\ot c)b[(aw_4\Box w_5)\Box cw_6]+f_1(b\ot a)f_1(c\ot b)c[b(aw_4\Box w_5)\Box w_6]
		\end{align*}
		and
		\begin{align*}
		w_1\Box (w_2\Box w_3)=&f_1(b\ot c)f_1(a\ot b)a[w_4\Box b(w_5\Box cw_6)]+f_1(b\ot c)f_1(b\ot a)b[aw_4\Box (w_5\Box cw_6)]\\
		+&f_1(c\ot b)f_1(a\ot c)a[w_4\Box c(bw_5\Box w_6)]+f_1(c\ot b)f_1(c\ot a)c[aw_4\Box (bw_5\Box w_6)].
		\end{align*}
		However 
		\[(w_4\Box bw_5)\Box cw_6=w_4\Box (bw_5\Box cw_6)=f_1(b\ot c)w_4\Box b(w_5 \Box cw_6)+f_1(c\ot b)w_4\Box c(bw_5 \Box w_6),\]
		\[aw_4\Box (bw_5\Box w_6)=(aw_4\Box bw_5)\Box w_6=f_1(a\ot b)a(w_4\Box bw_5) \Box w_6+f_1(b\ot a)b(aw_4\Box w_5) \Box w_6,\]
		and $f_1$ satisfies $f_1(x\ot y)f_1(y\ot z)\left(f_1(x\ot z)-1\right)=0$ for any set $\{x,y,z\}\in X$.
		Thus, $(w_1\Box w_2)\Box w_3=w_1\Box (w_2\Box w_3)$.
		\item[case $a=b$ and $(a\neq c)$.] By commutativity it is the same case as ($a=c$ and $b\neq a$) or ($b=c$ and $a\neq b$).
		\begin{align*}
		(w_1\Box w_2)\Box w_3=&f_1(a\ot a)f_1(a\ot c)a[(w_4\Box aw_5)\Box cw_6]+f_1(a\ot a)f_1(c\ot a)c[a(w_4\Box aw_5)\Box w_6]\\
		+&f_2(a\ot a)f_1(a\ot c)a[(aw_4\Box w_5)\Box cw_6]+f_2(a\ot a)f_1(c\ot a)c[a(aw_4\Box w_5)\Box w_6]
		\end{align*}
		and
		\begin{align*}
			w_1\Box (w_2\Box w_3)=&f_1(a\ot c)f_1(a\ot a)a[w_4\Box a(w_5\Box cw_6)]+f_1(a\ot c)f_2(a\ot a)a[aw_4\Box (w_5\Box cw_6)]\\
			+&f_1(c\ot a)f_1(a\ot c)a[w_4\Box c(aw_5\Box w_6)]+f_1(c\ot a)^2c[aw_4\Box (aw_5\Box w_6)].
		\end{align*}
		However 
		\[(w_4\Box aw_5)\Box cw_6=w_4\Box (aw_5\Box cw_6)=f_1(a\ot c)w_4\Box a(w_5 \Box cw_6)+f_1(c\ot a)w_4\Box c(aw_5 \Box w_6),\]
		\[aw_4\Box (aw_5\Box w_6)=(aw_4\Box aw_5)\Box w_6=f_1(a\ot a)a(w_4\Box aw_5) \Box w_6+f_2(a\ot a)a(aw_4\Box w_5) \Box w_6,\]
		and $f_1$ satisfies 
		\begin{enumerate}
			\item If $f_1(a\ot a)=f_2(a\ot a)\in\{0,1\}$ then
				\begin{enumerate}
					\item $f_1(a\ot b)f_1(b\ot a)[f_1(a\ot a)-1]=0$,
					\item $f_1(a\ot a)f_1(a\ot b)[f_1(a\ot b)-1]=0$,
					\item $f_1(a\ot a)f_1(b\ot a)[f_1(b\ot a)-1]=0$.
				\end{enumerate}
			\item If $f_1(a\ot a)=1-f_2(a\ot a)$ then $f_1(a\ot c)=1$ and $f_1(c\ot a)=0$.
		\end{enumerate}
		Thus, $(w_1\Box w_2)\Box w_3=w_1\Box (w_2\Box w_3)$.
		\item[case $a=b=c$ and $f_1(a\ot a)=f_2(a\ot a)$.] 
		\begin{align*}
		(w_1\Box w_2)\Box w_3=&f_1(a\ot a)^2a[(w_4\Box aw_5)\Box aw_6]+f_1(a\ot a)^2a[a(w_4\Box aw_5)\Box w_6]\\
		+&f_1(a\ot a)^2a[(aw_4\Box w_5)\Box aw_6]+f_1(a\ot a)^2a[a(aw_4\Box w_5)\Box w_6]
		\end{align*}
		and
		\begin{align*}
		w_1\Box (w_2\Box w_3)=&f_1(a\ot a)^2a[w_4\Box a(w_5\Box aw_6)]+f_1(a\ot a)^2a[aw_4\Box (w_5\Box aw_6)]\\
		+&f_1(a\ot a)^2a[w_4\Box a(aw_5\Box w_6)]+f_1(a\ot a)^2a[aw_4\Box (aw_5\Box w_6)].
		\end{align*}
		Thus, $(w_1\Box w_2)\Box w_3=w_1\Box (w_2\Box w_3)$.
		\item[case $a=b=c$ and $f_2(a\ot a)=1-f_1(a\ot a)$.] There exist three words $w_7$, $w_8$ and $w_9$ not necessarily non-empty, not starting by $a$ and three positive integers $k$, $l$ and $m$ such that $w_1=\underbrace{a\dots a}_{k \text{ times}}w_7$, $w_2=\underbrace{a\dots a}_{l \text{ times}}w_8$ and $w_3=\underbrace{a\dots a}_{m \text{ times}}w_9$. Besides, relations satisfied by $\Box$ enjoin $f_1(a\ot c)=1$ and $f_2(c\ot a)=0$ for any letter $c\neq a$. So,
		\begin{align*}
			(w_1\Box w_2) \Box w_3=&\Bigg[(\underbrace{a\dots a}_{k \text{ times}}\Box \underbrace{a\dots a}_{l \text{ times}})\Box \underbrace{a\dots a}_{k \text{ times}}  \Bigg]\Bigg[(w_7\Box w_8)\Box w_9\Bigg]\\
			=&\underbrace{a\dots a}_{k+l+m \text{ times}}\Bigg[(w_7\Box w_8)\Box w_9\Bigg]\\
			=&\Bigg[\underbrace{a\dots a}_{k \text{ times}}\Box (\underbrace{a\dots a}_{l \text{ times}}\Box \underbrace{a\dots a}_{k \text{ times}})  \Bigg]\Bigg[w_7\Box (w_8\Box w_9)\Bigg]=w_1\Box (w_2 \Box w_3).
		\end{align*}
	\end{description}
\end{proof}

\begin{cor}\label{corcarac}
	Let $\K$ be a field of characteristic 0, let $X$ be a countable alphabet and let $\Box$ be a weak shuffle product on $\K\langle X\rangle$.
	\begin{enumerate}
		\item There exists at most one letter $a$ such that $f_1(a\ot a)=1-f_2(a\ot a)$. 
		\item If there exists a letter $a$ such that $f_1(a\ot a)=1-f_2(a\ot a)$ then, for any word $u$ and $v$, the calculation of $u\Box v$ does not depend on the value of $f_1(a\ot a)$
		\item If $f_1(a\ot b)=f_1(b\ot a)=1$ then $f_1(a\ot a)=f_2(a\ot a)=f_1(b\ot b)=f_2(b\ot b)=1$, $f_1(a\ot c)=f_1(b\ot c)\in\{0,1\}$ and $f_1(c\ot a)=f_1(c\ot b)\in\{0,1\}$ for any $c\in X\setminus\{a,b\}$.
	\end{enumerate} 
\end{cor}
\begin{proof}
	\begin{enumerate}
		\item If there are two letters $a$ and $b$ such that $a\neq b$, $f_1(a\ot a)=1-f_2(a\ot a)$ and $f_1(b\ot b)=1-f_2(b\ot b)$ then $1=f_1(a\ot b)=0$ and $0=f_1(b\ot a) =1$. Contradiction.
		\item Let $a$ such that $f_1(a\ot a)=1-f_2(a\ot a)$. If $u$ and $v$ are words in $X^{*}\setminus aX^{*}$, since $f_1(a\ot b)=1$ and $f_1(b\ot a)=0$ for any $b\neq a$, there does not exist any triple $(w,u^{'},v^{'})$ such that $u\Box v=w(au^{'}\Box av^{'})$.
		\item If $f_1(a\ot b)=f_1(b\ot a)=1$ then, the fact that  $f_1(a\ot a)=f_2(a\ot a)=f_1(b\ot b)=f_2(b\ot b)=1$ comes directly from relations \ref{reltrois} given in Theorem \ref{carctwshuffle}. To prove $f_1(a\ot c)=f_1(b\ot c)\in\{0,1\}$ and $f_1(c\ot a)=f_1(c\ot b)\in\{0,1\}$ for any $c\in X\setminus\{a,b\}$, we use the relation $f_1(x\ot y)f_1(y\ot z)[f_1(x\ot z)-1]=0$ for any $x,y,z\in X$. 
	\end{enumerate}
\end{proof}
\begin{prop}\label{propisom}
	Let $\K$ be a field of characteristic 0, $X$ be a countable alphabet and $\Box$ a weak shuffle product on $\K\langle X\rangle$. We denote by $T$ the set $T=\{a\in X, f_1(a\ot a)\in\K\setminus\{0,1\}\}$. We assume $T\neq\emptyset$; so $T$ is a singleton $\{a\}$.
	Let $\Box'$ be the weak shuffle product defined by 
	\begin{itemize}
		\item $f_1^{'}(u\ot v)=f_1(u\ot u)$ for any $u\ot v\in X\ot X\setminus\{a\ot a\}$,
		\item $f_1^{'}(a\ot a)=1$ and $f_2^{'}(a\ot a)=1$.
	\end{itemize}
	Then, there exists an algebra isomorphism between  $(\K\langle X\rangle,\Box)$ and  $(\K\langle X\rangle,\Box')$.
\end{prop}

\begin{proof}
	Thanks to Corollary \ref{corcarac}, we know that the weak shuffle $\Box$ does not depend on the value of $f_1(a\ot a)$. We define $\psi:(\K\langle X\rangle,\Box)\rightarrow (\K\langle X\rangle,\Box^{'})$ by:
	\[\psi(w)=\begin{cases}
		w & \mbox{ if } w\notin aX^*,\\
		\frac{1}{n!}w & \mbox{ if } w=\underbrace{a\dots a}_{n \text{ times}}w_1 \text{ with } w_1\notin aX^*.
	\end{cases}\]
	
	Since $f_1(a\ot b)=1$ and $f_1(b\ot a)=0$ for any $b\in X\setminus \{a\}$, the linear map $\psi$ is an algebra morphism. It is trivially an isomorphism.
\end{proof}

\begin{prop}\label{valzeroun}
	Let $\K$ be a field of characteristic 0, let $X$ be an alphabet of cardinality 2 or 3 and let $\Box$ be a weak shuffle product on $\K\langle X\rangle$.
Let $\Box'$ be the weak shuffle product defined by 
\begin{itemize}
	\item $f_1^{'}(a\ot b)=1$ and $f_1^{'}(b\ot a)=0$ for any $(a\ot b)\in X\ot X$ such that $a\neq b$ and $f_1(a\ot b)\notin\{0,1\}$.
	\item $f_1^{'}(a\ot b)=f_1(a\ot b)$  for any $(a\ot b)\in X\ot X$ such that $a\neq b$ and $f_1(a\ot b)\in\{0,1\}$.
	\item $f_i^{'}(a\ot a)=f_i(a\ot a)$ for any $a\in X$ and any $i\in\{1,2\}$.
\end{itemize}
Then, there exists an algebra isomorphism between  $(\K\langle X\rangle,\Box)$ and  $(\K\langle X\rangle,\Box')$.
\end{prop}

\begin{proof}
	If $X=\{a,b\}$ then there is an one-parameter family of weak shuffle products $\Box$ such that $f_1(a\ot b)\notin \{0,1\}$. They are defined by $f_1(a\ot b)=k\in\K\setminus\{0,1\}$ and $f_1(b\ot a)=f_1(a\ot a)=f_2(a\ot a)=f_1(b\ot b)=f_2(b\ot b)=0$. We define $\Box^{'}$ by changing $k$ in 1. 
	 The map $\varphi$ defined by \[\varphi(w)=\begin{cases}
	\frac{1}{k^n}w & \mbox{ if } w=\underbrace{a\dots a}_{n \text{ times}}w{'} \text{ with } w{'}\in bX^*,\\
		w &\mbox{ else,}
	\end{cases}\] is an algebra isomorphism between  $(\K\langle X\rangle,\Box)$ and  $(\K\langle X\rangle,\Box')$ 
	
	Let us now consider the case $X=\{a,b,c\}$. Without loss of generality we assume $f_1(a\ot b)=k\in\K\setminus\{0,1\}$. The charactarisation of weak shuffle products given in Theorem \ref{carctwshuffle} leads to the following relations:
	\begin{itemize}
		\item $f_1(b\ot a)=f_1(a\ot a)=f_2(a\ot a)=f_1(b\ot b)=f_2(b\ot b)=0$,
		\item $f_1(a\ot c)f_1(c\ot a)=0$,
		\item $f_1(b\ot c)f_1(c\ot b)=0$,
		\item $f_1(a\ot c)f_1(c\ot b)=0$,
		\item $f_1(b\ot c)f_1(c\ot a)=0$,
		\item $f_1(u\ot v)f_1(v\ot w)[f_1(u\ot w)-1]=0$ where $\{u,v,w\}=X$.
	\end{itemize}
	Thus, the weak shuffle product $\Box$ is one of the following:
	\begin{enumerate}
		\item $f_1(a\ot c)=f_1(b\ot c)=f_1(c\ot a)=f_1(c\ot b)=0$ and $f_1(c\ot c)=f_2(c \ot c)\in\{0,1\}$. \label{casun}
		\item $f_1(a\ot c)=1$, $f_1(b\ot c)=p\in\K^*$ and $f_1(c\ot a)=f_1(c\ot b)=f_1(c\ot c)=f_2(c \ot c)=0$,
		\item $f_1(a\ot c)=1$, $f_1(b\ot c)=1$, $f_1(c\ot a)=f_1(c\ot b)=0$ and $f_1(c\ot c)=f_2(c \ot c)=1$, \label{fincasdeux}
		\item $f_1(a\ot c)=f_1(b\ot c)=0$, $f_1(c\ot a)=p\in\K^*$, $f_1(c\ot b)=1$ and $f_1(c\ot c)=f_2(c \ot c)=0$, \label{debutcastrois}
		\item $f_1(a\ot c)=f_1(b\ot c)=0$, $f_1(c\ot a)=1$, $f_1(c\ot b)=1$ and $f_1(c\ot c)=f_2(c \ot c)=1$,
		\item $f_1(a\ot c)=f_1(b\ot c)=0$, $f_1(c\ot a)=1$, $f_1(c\ot b)=1$ and $f_1(c\ot c)=1-f_2(c \ot c)$,
		\item $f_1(a\ot c)=f_1(b\ot c)=f_1(c\ot a)=0$, $f_1(c\ot b)=p\in\K^*$ and $f_1(c\ot c)=f_2(c \ot c)=0$,
		\item $f_1(a\ot c)=f_1(b\ot c)=f_1(c\ot a)=0$, $f_1(c\ot b)=1$ and $f_1(c\ot c)=f_2(c \ot c)=1$, \label{fincasquatre}
		\item $f_1(a\ot c)=p\in\K^*$, $f_1(b\ot c)=f_1(c\ot a)=f_1(c\ot b)=0$ and $f_1(c\ot c)=f_2(c \ot c)=0$, \label{debutcascinq}
		\item $f_1(a\ot c)=1$, $f_1(b\ot c)=f_1(c\ot a)=f_1(c\ot b)=0$ and $f_1(c\ot c)=f_2(c \ot c)=1$, \label{fincascinq}
	\end{enumerate}
	We define $\Box^{'}$ by $f_1^{'}(a\ot b)=1$ and $f_1^{'}(u\ot v)=f_1(u\ot v)$ if $u\ot v\neq a\ot b$.
	Let $\varphi_1$ and $\varphi_2$ be the maps defined by: for any word $w$,
	\[\varphi_1(w)=\begin{cases}
	\frac{1}{k^n}w & \mbox{ if } w=\underbrace{a\dots a}_{n \text{ times}}w^{'} \text{ with } w^{'}\in bX^*,\\
	w & \mbox{ else,} 
	\end{cases}\]
	
	and 
	
	\[\varphi_2(w)=\begin{cases}
	\frac{1}{k^{n_1+\dots + n_s}}w & \mbox{ if } w=\underbrace{c\dots c}_{q_1 \text{ times}}\underbrace{a\dots a}_{n_1 \text{ times}}\underbrace{c\dots c}_{q_2 \text{ times}}\dots \underbrace{c\dots c}_{q_s \text{ times}}\underbrace{a\dots a}_{n_s \text{ times}}\underbrace{c\dots c}_{q_{s+1} \text{ times}}w^{'} \text{ with } w^{'}\in bX^*\\
	& \mbox{ and } (q_1,\dots,q_{s+1})\in\N^{s+1},\\
	w & \mbox{ else.} 
	\end{cases}\]
	
	From case \ref{casun} to case \ref{fincasdeux} and from case \ref{debutcascinq} to case \ref{fincascinq} the map $\varphi_1$ is an algebra isomorphism between $(\K\langle X\rangle,\Box)$ and  $(\K\langle X\rangle,\Box')$. From case \ref{debutcastrois} to case \ref{fincasquatre}
	the map $\varphi_2$ is an algebra isomorphism between $(\K\langle X\rangle,\Box)$ and  $(\K\langle X\rangle,\Box')$.
	
	If maps $f_1^{'}$ and $f_2^{'}$ do not take their values in $\{0,1\}$ we apply the previous process once again to $\Box^{'}$. And then, we find a weak shuffle product $\Box^{''}$ such that $f_1{''}(u\ot v), f_2{''}(u\ot v)\in\{0,1\}$ for any $(u\ot v)\in X\ot X$.  
\end{proof}

\begin{conj}\label{conj}
	Proposition \ref{valzeroun} is still true for any countable alphabet.
\end{conj}
\Rq{1} If $X$ is an alphabet such that $\{a,b,c,d\}\subset X$ and $f_1(a\ot b)\notin\{0,1\}$ then relations 
\begin{enumerate}
	\item $f_1(a\ot x)f_1(x\ot a)=0$,
	\item $f_1(b\ot x)f_1(x\ot b)=0$,
	\item $f_1(a\ot x)f_1(x\ot b)=0$,
	\item $f_1(b\ot x)f_1(x\ot a)=0$,
\end{enumerate}
are still satisfied for any letter $x\in X$. However, if $x,y\in X\setminus\{a,b\}$, even if they satisfy relations given in Theorem \ref{carctwshuffle}, it is hard to anticipate the part of $x$ facing $y$. 

\subsection{Weak shuffle products on $\K\langle\{a,b\}\rangle$} 
Let $X=\{a,b\}$ be an alphabet of cardinality $2$. By using the characterisation given in Theorem \ref{carctwshuffle}, there are 10 families of weak shuffle products defined on $\K\langle X\rangle$.
Let $C$ be the $6$-tuple $C=\Big(f_1(a\ot b),f_1(b\ot a), f_1(a\ot a), f_2(a\ot a), f_1(b\ot b), f_2(b\ot b)\Big)$. If $k\in\K^*$ and $\alpha\in\K$ then $C$ is one of the following $6$-tuples
\begin{align*}
C_1=   &(0,0,0,0,0,0),             & C_2=&(k,0,0,0,0,0),            & C_3=&(1,0,1,1,0,0),\\
C_4=   &(1,0,0,0,1,1),             & C_5=&(0,0,1,1,0,0),            & C_6=&(0,0,1,1,1,1),\\
C_7=   &(1,0,\alpha,1-\alpha,0,0), & C_8=&(1,0,\alpha,1-\alpha,1,1),& C_9=&(1,0,1,1,1,1),\\
C_{10}=&(1,1,1,1,1,1).             &     &                          &     &
\end{align*}
For any $n\in\llbracket 1,10\rrbracket$, we denote by $\underset{n}{\Box}$ the weak shuffle product associated to $C_n$. The concatenation of two words  $u$ and $v$ is denoted by  $uv$. The empty word is denoted by $1$.
\begin{description}
	\item[Case $n=2$.] Thanks to Proposition \ref{valzeroun}, for any $k\in\K^*$ the weak shuffle product defined by $C_2$ is isomorphic to the case $(1,0,0,0,0,0)$. Let $u$ and $v$ be two non-empty words. Then 
	\[u\underset{2}{\Box} v=\begin{cases}
	k^nuv &\mbox{ if } (u=\underbrace{a\dots a}_{n \text{ times}} \text{ and } v=bw \text{ with } w\in X^*) \\
	k^nvu &\text{ if } (v=a\dots a \text{ and } u	=bw \text{ with } w\in X^*),\\
	0 & \mbox{ else.}
	\end{cases}\]
	
	\item[Cases $n=3$ and $n=7$.] Thanks to Proposition \ref{propisom} the weak shuffle products defines by $C_3$ and $C_7$ are isomorphic. Let $u$ and $v$ be two non-empty words. Then 
	\[ u\underset{3}{\Box} v=\begin{cases}
	uv &\mbox{ if } (u=a\dots a \text{ and } v=bw \text{ with } w\in X^*) \\
	vu &\text{ if } (v=a\dots a \text{ and } u	=bw \text{ with } w\in X^*),\\
	\displaystyle\binom{k+l}{k}\underbrace{a\dots a}_{k+l \text{ times}}w  &\mbox{ if } (u=\underbrace{a\dots a}_{k \text{ times}} \text{ and } v=\underbrace{a\dots a}_{l \text{ times}}w \text{ with } w\in bX^*\cup\{1\}) \\ &\text{ or } (v=\underbrace{a\dots a}_{k \text{ times}} \text{ and } u=\underbrace{a\dots a}_{l \text{ times}}w \text{ with } w\in bX^*\cup\{1\}),\\
	0 & \mbox{ else,}
	\end{cases}\]
	
	and $u\underset{7}{\Box} v=\begin{cases}
	uv &\mbox{ if } (u=a\dots a \text{ and } v=bw \text{ with } w\in X^*) \\ 
	vu &\text{ if } (v=a\dots a \text{ and } u	=bw \text{ with } w\in X^*),\\
	\underbrace{a\dots a}_{k+l \text{ times}}w  &\mbox{ if } (u=\underbrace{a\dots a}_{k \text{ times}} \text{ and } v=\underbrace{a\dots a}_{l \text{ times}}w \text{ with } w\in bX^*\cup\{1\}) \\ &\text{ or } (v=\underbrace{a\dots a}_{k \text{ times}} \text{ and } u=\underbrace{a\dots a}_{l \text{ times}}w \text{ with } w\in bX^*\cup\{1\}),\\
	0 & \mbox{ else.}
	\end{cases}$
	
	\item[Case $n=5$.] Let $u$ and $v$ be two non-empty words. Then 
	\[ u\underset{5}{\Box} v=\begin{cases}
	\displaystyle\binom{k+l-1}{k}\underbrace{a\dots a}_{k+l \text{ times}}w  &\mbox{ if } (u=\underbrace{a\dots a}_{k \text{ times}} \text{ and } v=\underbrace{a\dots a}_{l \text{ times}}w \text{ with } w\in bX^*) \\ &\text{ or } (v=\underbrace{a\dots a}_{k \text{ times}} \text{ and } u=\underbrace{a\dots a}_{l \text{ times}}w \text{ with } w\in bX^*),\\
	\displaystyle\binom{k+l}{k}\underbrace{a\dots a}_{k+l \text{ times}}  &\mbox{ if } u=\underbrace{a\dots a}_{k \text{ times}} \text{ and } v=\underbrace{a\dots a}_{l \text{ times}},\\
	0 & \mbox{ else.}
	\end{cases}\]
	
	\item[Case $n=6$.] Let $u$ and $v$ be two non-empty words. Then 
	\[ u\underset{6}{\Box} v=\begin{cases}
	\displaystyle\binom{k+l-1}{k}\underbrace{a\dots a}_{k+l \text{ times}}w  &\mbox{ if } (u=\underbrace{a\dots a}_{k \text{ times}} \text{ and } v=\underbrace{a\dots a}_{l \text{ times}}w \text{ with } w\in bX^*) \\ &\text{ or } (v=\underbrace{a\dots a}_{k \text{ times}} \text{ and } u=\underbrace{a\dots a}_{l \text{ times}}w \text{ with } w\in bX^*),\\
	\displaystyle\binom{k+l}{k}\underbrace{a\dots a}_{k+l \text{ times}}  &\mbox{ if } u=\underbrace{a\dots a}_{k \text{ times}} \text{ and } v=\underbrace{a\dots a}_{l \text{ times}},\\
	\displaystyle\binom{k+l-1}{k}\underbrace{b\dots b}_{k+l \text{ times}}w  &\mbox{ if } (u=\underbrace{b\dots b}_{k \text{ times}} \text{ and } v=\underbrace{b\dots b}_{l \text{ times}}w \text{ with } w\in aX^*) \\ &\text{ or } (v=\underbrace{b\dots b}_{k \text{ times}} \text{ and } u=\underbrace{b\dots b}_{l \text{ times}}w \text{ with } w\in aX^*),\\
	\displaystyle\binom{k+l}{k}\underbrace{b\dots b}_{k+l \text{ times}}  &\mbox{ if } u=\underbrace{b\dots b}_{k \text{ times}} \text{ and } v=\underbrace{b\dots b}_{l \text{ times}},\\
	0 & \mbox{ else.}
	\end{cases}\]
	\item[Case $n=4$.] First, it is natural to ask whether or not this case is isomorphic to the case with $n=3$? In fact, not. A counter-example is given by the elements $u$ of degree 2 such that $u^2=0$. Indeed, 
	\begin{enumerate}
		\item with the case $n=4$, if $u=\lambda aa+\mu bb+\sigma ab +\tau ba$ then 
		\begin{align*}
			u^2=&6\mu^2bbbb+2\tau^2baba+2\lambda\mu aabb+2\lambda\tau aaba+6\mu\sigma abbb \\
			+& 2\mu\tau (babb+bbab+bbba) + 2\sigma\tau (abab+abba). 
		\end{align*}
		So $u^2=0 \iff \mu=\tau=0$ and $\bigg\{u\in \K\langle\{a,b\}^*\rangle, \length(u)=2 \text{ and } u^2=0 \bigg\}=\Span(aa, ab)$.
		\item with the case $n=3$, if $u=\lambda aa+\mu bb+\sigma ab +\tau ba$ then 
		\begin{align*}
		u^2=&6\lambda^2aaaa+2\lambda\mu aabb+6\lambda\sigma aaab+2\lambda\tau aaba. 
		\end{align*}
		So $u^2=0 \iff \lambda=0$  and $\bigg\{u\in \K\langle\{a,b\}^*\rangle, \length(u)=2 \text{ and } u^2=0 \bigg\}=\Span(bb, ab, ba)$.
	\end{enumerate}
	
	Let $u$ and $v$ be two non-empty words. Then 
	\begin{enumerate}
		\item If $u=\underbrace{a \dots a}_{m \text{ times}}u^{'}$  and $u{'},v\in bX^{*}\cup\{1\}$ then \[u\underset{4}{\Box} v=v\underset{4}{\Box} u=\underbrace{a \dots a}_{m \text{ times}}(u^{'}\underset{4}{\Box}v).\]
		\item If $ u=\underbrace{b\dots b}_{m_1 \text{ times}}u^{'}$, $v=\underbrace{b\dots b}_{m_2 \text{ times}}v^{'}$ and $ u^{'},v^{'}\in aX^{*}\cup\{1\}$ then 
		\begin{align*}
		 u\underset{4}{\Box}v=&\displaystyle\sum_{k=0}^{m_2-1}\binom{m_1+k-1}{k}\underbrace{b\dots b}_{m_1+k \text{ times}}(u^{'}\underset{4}{\Box}\underbrace{b\dots b}_{m_2-k \text{ times}}w^{'})\\
		 &+\displaystyle\sum_{k=0}^{m_1-1}\binom{m_2+k-1}{k}\underbrace{b\dots b}_{m_2+k \text{ times}}(\underbrace{b\dots b}_{m_1-k \text{ times}}u^{'}\underset{4}{\Box}v^{'})\\
		 =&v\underset{4}{\Box}u
		\end{align*}
		\item If $u,v\in aX^{*}$ then $u\underset{4}{\Box}v=v\underset{4}{\Box}u=0$.
	\end{enumerate}
	
	\item[Cases $n=8$ and $n=9$.] We recall that the case $n=8$ does not depend on $\alpha\in\K$. Thanks to Proposition \ref{propisom}  the weak shuffle products defined by $C_8$ and $C_9$ are isomorphic. Let $u$ and $v$ be two non-empty words. Then 
	\begin{enumerate}
		\item If $u=\underbrace{a \dots a}_{m \text{ times}}u^{'}$  and $u{'},v\in bX^{*}\cup\{1\}$ then \[u\underset{9}{\Box} v=v\underset{9}{\Box} u=\underbrace{a \dots a}_{m \text{ times}}(u^{'}\underset{9}{\Box}v)=u\underset{8}{\Box} v=v\underset{8}{\Box} u.\]
		\item If $ u=\underbrace{b\dots b}_{m_1 \text{ times}}u^{'}$, $v=\underbrace{b\dots b}_{m_2 \text{ times}}v^{'}$ and $ u^{'},v^{'}\in aX^{*}\cup\{1\}$ then 
		\begin{align*}
		u\underset{9}{\Box}v=&\displaystyle\sum_{k=0}^{m_2-1}\binom{m_1+k-1}{k}\underbrace{b\dots b}_{m_1+k \text{ times}}(u^{'}\underset{9}{\Box}\underbrace{b\dots b}_{m_2-k \text{ times}}w^{'})\\
		&+\displaystyle\sum_{k=0}^{m_1-1}\binom{m_2+k-1}{k}\underbrace{b\dots b}_{m_2+k \text{ times}}(\underbrace{b\dots b}_{m_1-k \text{ times}}u^{'}\underset{9}{\Box}v^{'})\\
		=&v\underset{9}{\Box}u=u\underset{8}{\Box}v=v\underset{8}{\Box}u.
		\end{align*}
		\item If $u=\underbrace{a\dots a}_{k \text{ times}}u^{'}$, $v=\underbrace{a\dots a}_{l \text{ times}}v^{'}$  and $u^{'},v^{'}\in bX^{*}\cup\{1\}$ then \[u\underset{9}{\Box}v=v\underset{9}{\Box}u=\binom{k+l}{k}\underbrace{a\dots a}_{k+l \text{ times}}(u^{'}\underset{9}{\Box}v^{'}),\] and
		\[u\underset{8}{\Box}v=v\underset{8}{\Box}u=\underbrace{a\dots a}_{k+l \text{ times}}(u^{'}\underset{8}{\Box}v^{'}).\]
	\end{enumerate}
	From the previous calculations, we have the following consequence:
	\begin{cor}
		Let $v$ and $w$ be two words. Then $v\underset{9}{\Box} w\neq0$.
	\end{cor}
	
	\Rq{1} For cases $n\in\{4,8,9\}$, since $f_1(a\ot b)=1$ and $f_1(b\ot a)=0$, the calculation of $ u\underset{n}{\Box} v$ where $u=\underbrace{b\dots b}_{m_1 \text{ times}}u^{'}$, $v=\underbrace{b\dots b}_{m_2 \text{ times}}v^{'}$ and  $u^{'},v^{'}\in aX^{*}\cup\{1\}$ does not depend on the values of $f_1(a\ot a)$ nor $f_2(a\ot a)$. We give the value of $u\underset{4}{\Box}v(=u\underset{8}{\Box}v=u\underset{9}{\Box}v)$ for some example couple $(u,v)\in(bX^{*})^2$. For some examples of pairs $(x,p)\in X\times\N^*$, to lighten the notation, we write $x^p$ instead of $\underbrace{x \dots x}_{p \text{ times}}$.
	
	Let $(m,s,p,r)$ be a quadruple of positive integers. Then:
	\[b^ma^s\underset{4}{\Box}b^pa^r=\displaystyle\sum_{k=0}^{p-1}\binom{m+k-1}{k}b^{m+k}a^sb^{p-k}a^r+\sum_{k=0}^{m-1}\binom{p+k-1}{k}b^{p+k}a^rb^{m-k}a^s.\]
	Let $(m,s,p,r,t)$ be a quintuple of positive integers such that $m\geq 2$. Then:
	\begin{align*}
	b^ma^s\underset{4}{\Box}b^pa^rb^t=&\sum_{k=0}^{p-1}\binom{m+k-1}{k}b^{m+k}a^sb^{p-k}a^rb^t+\displaystyle\sum_{k=0}^{t}\binom{m+k-1}{k}b^pa^rb^{m+k}a^sb^{t-k}\\
	&+\sum_{\substack{f+g= m\\f\in\N^*\\g\in\N^*}}\displaystyle\sum_{k=0}^{t}\binom{f+p-1}{f}\binom{g+k-1}{k}b^{p+f}a^rb^{g+k}a^sb^{t-k}.
	\end{align*}
\end{description}

\begin{prop}\label{Cneuf}
	Let $\underset{9}{\Box}$ be the weak shuffle product defined by $C_{9}$. Let $p$ be a positive integer and $n\in\{1,2,3\}$. We denote by $K_{(n,p)}$ the set 
	\[K_{(n,p)}=\Bigg\{u=\sum_{\substack{w\in X^{*}\\ \length(w)=n}}\lambda_w w, u^p=0\Bigg\}.\]
	Then, $K_{(n,p)}=\{0\}$.
\end{prop}
\begin{proof}
	We equip $X^*$ with the lexicographic order. For any words $v$ and $w$ we denote by $\max(v\Box w)$ the greatest word of length $l=\length(v)+\length(w)$ which appears in $v\Box w$ for the lexicographic order.

	 If $u=\displaystyle\sum_{\substack{w\in X^{*}\\ \length(w)=n}}\lambda_w w$ then
	\[u^p=\sum_{\substack{w\in X^{*}\\ \length(w)=n}}\lambda_w^p (w\underset{9}{\Box} \dots \underset{9}{\Box} w)+\sum_{l=2}^{\min(p,x_n)}\sum_{(\alpha_1,\dots,\alpha_l)\models p}\sum_{\substack{w_1<\dots <w_l\in X^{*}\\ \forall i \length(w_i)=n}}\lambda_{w_1}^{\alpha_1}\dots \lambda_{w_l}^{\alpha_l}(w_1\underset{9}{\Box} \dots \underset{9}{\Box} w_l).\]
	\begin{enumerate}
		\item If $n=1$ then the result is trivial.
		\item If $n=2$ then 
		\begin{align*}
			aa^p&=\frac{(2p)!}{2^p} \underbrace{a \dots a}_{2p\text{ times}}, &
			ab^p&=(p!)^2 \underbrace{a \dots a}_{p\text{ times}}\underbrace{b \dots b}_{p\text{ times}},&
			ba^p&=p! \underbrace{ba \dots ba}_{p\text{ times}},&
			bb^p&=\frac{(2p)!}{2^p}\underbrace{b \dots b}_{2p\text{ times}},
		\end{align*}
		and \[\max(aa^k\underset{9}{\Box} ab^l\underset{9}{\Box} ba^m\underset{9}{\Box} bb^n)=\underbrace{a\dots a}_{\substack{2k+l \\ \text{ times}}}\underbrace{b \dots b}_{\substack{2n+l \\ \text{ times}}}\underbrace{ba \dots ba}_{\substack{m \\ \text{ times}}}.\]
		Thus $\lambda_{aa}=\lambda_{bb}=\lambda_{ba}=\lambda_{ba}=0$.
		\item If $n=3$ then 
		\begin{align*}
		w_1=&aaa^p=\frac{(3p)!}{(3!)^p} \underbrace{a \dots a}_{3p\text{ times}}, &
		w_2=&aab^p=\frac{(2p)!p!}{2^p} \underbrace{a \dots a}_{2p\text{ times}}\underbrace{b \dots b}_{p\text{ times}},\\
		w_3=&aba^p=(p!)^2 \underbrace{a \dots a}_{p\text{ times}}\underbrace{ba \dots ba}_{p\text{ times}},&
		w_4=&abb^p=\frac{(2p)!p!}{2^p} \underbrace{a \dots a}_{p\text{ times}}\underbrace{b \dots b}_{2p\text{ times}},\\
		w_5=&baa^p=p! \underbrace{baa \dots baa}_{p\text{ times}},&
		w_6=&bbb^p=\frac{(3p)!}{(3!)^p} \underbrace{b \dots b}_{3p\text{ times}}.
		\end{align*}
		For $bab^p$ and $bba^p$, there are several terms in the result. For $bab^p$ we will use $w_7=\underbrace{bab \dots bab}_{p \text{ times}}$ and, for $bba^n$ we will use $w_8=\underbrace{b\dots b}_{p \text{ times}}\underbrace{ba \dots ba}_{p \text{ times}}$. In fact, for the lexicographic order,  we use the maximal term obtained in each product. For any $i$ determine how build $w_i$ by doing the weak shuffle of $p$ words of length $3$. We get $\lambda_{aaa}=\lambda_{bbb}=\lambda_{aba}=\lambda_{baa}=\lambda_{aab}=\lambda_{abb}=\lambda_{bab}=\lambda_{bba}=0$.
	\end{enumerate}
\end{proof}
\begin{conj}\label{conjKnp}
	Let $\underset{9}{\Box}$ be the weak shuffle product defined by $C_{9}$. For any positive integers $p$ and $n$, 
	we have $K_{(n,p)}=\{0\}$.
\end{conj}
\Rq{}
	\begin{enumerate}
		\item By induction we can express $\max(u\underset{9}{\Box} v)$ for any words $u$ and $v$.
		\begin{description}
			\item[Case $w_1$ and $w_2$ are in $aX^*$.] There exist $\alpha,\beta\in\N^*$ and $w_1^{'}, w_2^{'}\in bX^*\cup\{1\}$ such that $w_1=\underbrace{a \dots a}_{\alpha \text{ times}}w_1^{'}$ and $w_2=\underbrace{a \dots a}_{\beta \text{ times}}w_2^{'}$. Then, 
			\[\max(w_1\underset{9}{\Box} w_2)=\underbrace{a \dots a}_{\alpha+\beta \text{ times}}\max(w_1^{'}\underset{9}{\Box} w_2^{'}).\] 
			\item[Case $w_1\in aX^*$ and $w_2\in bX^*$.] There exist $\alpha\in\N^*$ and $w_1^{'}\in bX^*\cup\{1\}$ such that $w_1=\underbrace{a \dots a}_{\alpha \text{ times}}w_1^{'}$. Then, 
			\[\max(w_1\underset{9}{\Box} w_2)=\underbrace{a \dots a}_{\alpha \text{ times}}\max(w_1^{'}\underset{9}{\Box} w_2).\] 
			\item[Case $w_1$ and $w_2$ are in $bX^*$.] There exist $\alpha,\beta\in\N^*$,  $p,q\in\N$ (they are not necessarily different from 0) and $w_1^{'}, w_2^{'}\in bX^*\cup\{1\}$ such that $w_1=\underbrace{b \dots b}_{\alpha \text{ times}}\underbrace{a \dots a}_{p \text{ times}}w_1^{'}$ and $w_2=\underbrace{b \dots b}_{\beta \text{ times}}\underbrace{a \dots a}_{q \text{ times}}w_2^{'}$. Thus, 
			\begin{itemize}
				\item If $0<q<p$  then 
				\[\max(w_1\underset{9}{\Box} w_2)=\underbrace{b \dots b}_{\alpha+\beta-1 \text{ times}}\underbrace{a \dots a}_{q \text{ times}}\max(b\underbrace{a \dots a}_{p \text{ times}}w_1^{'}\underset{9}{\Box} w_2^{'}).\] 
				\item If $0<p<q$ then 
				\[\max(w_1\underset{9}{\Box} w_2)=\underbrace{b \dots b}_{\alpha+\beta-1 \text{ times}}\underbrace{a \dots a}_{p \text{ times}}\max(w_1^{'}\underset{9}{\Box} b\underbrace{a \dots a}_{q \text{ times}}w_2^{'}).\] 
				\item If $0<p$ and $p=q$ then $\max(w_1\underset{9}{\Box} w_2)=\max(\tilde{w}_1,\tilde{w}_2)$ where
				\[\tilde{w}_1=\underbrace{b \dots b}_{\alpha+\beta-1 \text{ times}}\underbrace{a \dots a}_{q \text{ times}}\max(b\underbrace{a \dots a}_{p \text{ times}}w_1^{'}\underset{9}{\Box} w_2^{'})\]
				and
				\[\tilde{w}_2=\underbrace{b \dots b}_{\alpha+\beta-1 \text{ times}}\underbrace{a \dots a}_{p \text{ times}}\max(w_1^{'}\underset{9}{\Box} b\underbrace{a \dots a}_{q \text{ times}}w_2^{'})).\] 
				\item If $p=0$ (respectively $q=0$) then $w_1=\underbrace{b \dots b}_{\alpha \text{ times}}$ (respectively $w_2=\underbrace{b \dots b}_{\beta \text{ times}}$) and
				\[\max((w_1\underset{9}{\Box} w_2))=w_1w_2 (\text{ respectively }  \max((w_1\underset{9}{\Box} w_2))=w_2w_1).\] 
			\end{itemize}
		\end{description}
		For instance, \begin{align*}
			\max(ab\underset{9}{\Box} abaa)&=aa\max(b\underset{9}{\Box} baa)=aabbaa,\\
			\max(bba \underset{9}{\Box} baa)&=bbabaa,\\
			\max(bbbaaabba\underset{9}{\Box}bbaabbba)&=bbbbaa\max(baaabba\underset{9}{\Box}bbba)=bbbbaabbbabaaabba.
		\end{align*}
		\item  For $p=2$ Conjecture \ref{conjKnp} is implied by the statement  "Let $n$ be a positive integer, let $w_1$, $w_2$ and $w$ be three non-empty words of length $n$ such that $w_1\leq w_2\leq w$ and $w_1<w$. Then $\max(w_1\underset{9}{\Box}w_2)<\max(w\underset{9}{\Box} w)$." We attend a reasoning by induction but there are some obstructions. Indeed, it leads us to compare $\max(u_1\underset{9}{\Box}u_2)$ and $\max(u_3\underset{9}{\Box}u_4)$ where $u_1\leq u_3$, $u_2\leq u_4$, $\length(u_1)=\length(u_3)$, $\length(u_2)=\length(u_4)$ and $(u_1,u_2)\neq(u_3,u_4)$. Then, it leads us to determine if  $\max(v_1\underset{9}{\Box}v_2)>\max(v_3\underset{9}{\Box}v_4)$ or $\max(v_1\underset{9}{\Box}v_2)<\max(v_3\underset{9}{\Box}v_4)$ where $v_1< v_3$, $v_2>v_4$. If we consider $v_1=a$, $v_2=bb$, $v_3=ab$ and $v_4=b$, then we get  $\max(v_1\underset{9}{\Box}v_2)=abb=\max(v_3\underset{9}{\Box}v_4)$.
	\end{enumerate}	
By using computation programs realised with {\ttfamily Maxima}, (\emph{c.f.} Section \ref{computation}) we get:
\begin{lemma}\label{lemmaprogramme}
	Let $n$ be a positive integer smaller than or equal to $7$. Then $K_{n,2}=\{0\}$.
\end{lemma}

\begin{prop}
	Let $X$ be the alphabet $\{a,b\}$ and $\mathcal{S}$ be the set defined by $\mathcal{S}=\{C_1\dots C_{10}\}$ equipped with the relation $\equiv$ such that:
	for any $A$ and $B$ in $\mathcal{S}$, $A\equiv B$ if and only if there exists an homogenous isomorphism between $(\K\langle X\rangle,\Box_A)$ and $(\K\langle X\rangle,\Box_B)$ where $\Box_A$ (respectively $\Box_B$) is the shuffle product associated to $A$ (respectively $B$). Let $n$ be the number of isomorphic classes.
	
	Then $n\in\{7,8\}$.
\end{prop}

\subsection{Weak shuffle products on $\K\langle\{a,b,c\}\rangle$} 
Let $X=\{a,b,c\}$ be an alphabet of cardinality $3$. 
Let $C$ be the $12$-tuple $C=\Big(f_1(a\ot b),f_1(b\ot a),f_1(b\ot c), f_1(c\ot b), f_1(a\ot c), f_1(c\ot a) f_1(a\ot a), f_2(a\ot a), f_1(b\ot b), f_2(b\ot b),f_1(c\ot c), f_2(c\ot c)\Big)$. By using Theorem \ref{carctwshuffle}, if $(k,m)\in(\K^*)^2$ and $\alpha\in\K$ then $C$ is one of the following tuples
\begin{align*}
C_1 =  &(0,0,0,0,0,0,0,0,0,0,0,0),             & C_2   =&(0,0,0,0,0,0,1,1,0,0,0,0),\\
C_3 =  &(0,0,0,0,0,0,1,1,1,1,0,0),             & C_4   =&(0,0,0,0,0,0,1,1,1,1,1,1),\\
C_5 =  &(k,0,0,0,0,0,0,0,0,0,0,0),             & C_6   =&(k,0,0,0,0,0,0,0,0,0,1,1),\\
C_7 =  &(1,0,0,0,0,0,1,1,0,0,0,0),             & C_8   =&(1,0,0,0,0,0,0,0,1,1,0,0),\\
C_9 =  &(1,0,0,0,0,0,1,1,1,1,0,0),             & C_{10}=&(1,0,0,0,0,0,1,1,0,0,1,1),\\
C_{11}=&(1,0,0,0,0,0,0,0,1,1,1,1),             & C_{12}=&(1,0,0,0,0,0,1,1,1,1,1,1),\\
C_{13}=&(1,1,0,0,0,0,1,1,1,1,0,0),             & C_{14}=&(1,1,0,0,0,0,1,1,1,1,1,1),\\
C_{15}=&(k,0,0,m,0,0,0,0,0,0,0,0),             & C_{16}=&(k,0,0,1,0,0,0,0,0,0,1,1),\\
C_{17}=&(1,0,0,1,0,0,0,0,1,1,0,0),             & C_{18}=&(1,0,0,1,0,0,1,1,1,1,0,0),\\
C_{19}=&(1,0,0,1,0,0,1,1,0,0,1,1),             & C_{20}=&(1,0,0,1,0,0,1,1,1,1,1,1),\\
C_{21}=&(k,0,0,0,m,0,0,0,0,0,0,0),             & C_{22}=&(1,0,0,0,1,0,1,1,0,0,0,0),\\
C_{23}=&(1,0,0,0,1,0,\alpha,1-\alpha,0,0,0,0), & C_{24}=&(k,0,0,0,1,0,0,0,0,0,1,1),\\
C_{25}=&(1,0,0,0,1,0,1,1,1,1,0,0),             & C_{26}=&(1,0,0,0,1,0,\alpha,1-\alpha,1,1,0,0),\\
C_{27}=&(1,0,0,0,1,0,0,0,1,1,1,1),             & C_{28}=&(1,0,0,0,1,0,1,1,1,1,1,1),\\
C_{29}=&(1,0,0,0,1,0,\alpha,1-\alpha,1,1,1,1), & C_{30}=&(k,0,m,0,1,0,0,0,0,0,0,0),\\
C_{31}=&(1,0,k,0,1,0,1,1,0,0,0,0),             & C_{32}=&(1,0,1,0,1,0,0,0,1,1,0,0),\\
C_{33}=&(k,0,1,0,1,0,0,0,0,0,1,1),             & C_{34}=&(1,0,1,0,1,0,1,1,1,1,0,0),\\
C_{35}=&(1,0,1,0,1,0,1,1,0,0,1,1),             & C_{36}=&(1,0,1,0,1,0,0,0,1,1,1,1),\\
C_{37}=&(1,0,1,0,1,0,1,1,1,1,1,1),             & C_{38}=&(1,0,k,0,1,0,\alpha,1-\alpha,0,0,0,0),\\
C_{39}=&(1,0,1,0,1,0,\alpha,1-\alpha,1,1,0,0), & C_{40}=&(1,0,1,0,1,0,\alpha,1-\alpha,0,0,1,1),\\
C_{41}=&(1,0,1,0,1,0,\alpha,1-\alpha,1,1,1,1), & C_{42}=&(1,1,1,0,1,0,1,1,1,1,0,0),\\
C_{43}=&(1,1,1,0,1,0,1,1,1,1,1,1)\},           & C_{44}=&(1,1,0,1,0,1,1,1,1,1,0,0),\\
C_{45}=&(1,1,0,1,0,1,1,1,1,1,1,1),             & C_{46}=&(1,1,0,1,0,1,1,1,1,1,\alpha,1-\alpha),\\
C_{47}=&(1,1,1,1,1,1,1,1,1,1,1,1).             &        & 
\end{align*}
\begin{prop}
	Let $X$ be the alphabet $\{a,b,c\}$ and $\mathcal{S}$ be the set defined by $\mathcal{S}=\{C_1\dots C_{47}\}$ equipped with the relation $\equiv$ such that:
	for any $A$ and $B$ in $\mathcal{S}$, $A\equiv B$ if and only if there exists an homogenous isomorphism between $(\K\langle X\rangle,\Box_A)$ and $(\K\langle X\rangle,\Box_B)$ where $\Box_A$ (respectively $\Box_B$) is the shuffle product associated to $A$ (respectively $B$). Let $n$ be the number of  isomorphic classes.

	Then $n\in\llbracket 33, 39 \rrbracket$.
\end{prop}

\begin{proof}
	Thanks to Proposition \ref{valzeroun}, in any set, it is sufficient to consider that $k=m=1$.
	Thanks to Proposition \ref{propisom}, we can prove that cases $C_{22}$ and $C_{23}$ are isomorphic, cases $C_{25}$ and $C_{26}$ are isomorphic, cases $C_{28}$ and $C_{29}$ are isomorphic, cases $C_{31}$ and $C_{38}$ are isomorphic, cases $C_{34}$ and $C_{39}$ are isomorphic, cases $C_{35}$ and $C_{40}$ are isomorphic, cases $C_{37}$ and $C_{41}$ are isomorphic and cases $C_{45}$ and $C_{46}$ are isomorphic.	

	Let $K_1$, $K_2$ and $K_3$ be the sets defined by:
	\begin{itemize}
		\item $K_1=\bigg\{u=\displaystyle\sum_{x\in X}\lambda_x x,~u^2=0\bigg\}$,
		\item $K_2=\bigg\{u=\displaystyle\sum_{\substack{w\in X^*\\ \length(w)=2}}\lambda_w w,~u^2=0\bigg\}$,
		\item $K_3=\bigg\{u=\displaystyle\sum_{\substack{w\in X^*\\ \length(w)=3}}\lambda_w w,~u^2=0\bigg\}$.
	\end{itemize}
	By using $K_1$ and $K_2$, we conclude that $C_6$, $C_7$ and $C_8$ are in three different isomorphic classes, $C_9$, $C_{10}$ and $C_{11}$ are in three different isomorphic classes,  $C_{16}$, $C_{17}$, $C_{22}$ and $C_{24}$ are in four different isomorphic classes, $C_{18}$, $C_{19}$, $C_{25}$ and $C_{27}$ are in four different isomorphic classes, $C_{15}$ and $C_{21}$ are in two different isomorphic classes, $C_{31}$, $C_{32}$ and $C_{33}$ are in three different isomorphic classes, $C_{34}$, $C_{35}$ and $C_{36}$ are in three different isomorphic classes, $C_{42}$ and $C_{44}$ are in two different isomorphic classes. With $K_3$, we prove that $C_{20}$ and $C_{28}$ are in two different isomorphic classes.
	Those sets do not enable us to conclude if there exists an isomorphism between $C_{9}$ and $C_{13}$, between $C_{12}$ and  $C_{14}$, between $C_{34}$ and  $C_{42}$, between $C_{36}$ and  $C_{44}$, between $C_{43}$ and  $C_{47}$, between $C_{45}$ and  $C_{47}$.
\end{proof}
\section{Weak shuffle algebras, dendriform algebras, quadri-algebras}
Dendriform algebras \cite{Loday2001} and quadri-algebras  \cite{Aguiar2004} are algebraic structures which enables one to split the associativity. Actually, a dendriform algebra is an algebra $\mathcal{A}$ equipped with a left product $\prec$ and a right product $\succ$  making the couple $(\mathcal{A},\prec+\succ)$ into an associative algebra and satisfying compatibilities. A quadri-algebra is obtained by splitting each product of a dendriform algebra in two products and the four new products must respect compatibilities. So, a quadri-algebra leads to two dendriform structures and the sum of the four products gives an associative product.

Those two notions have been extensively studied. For instance, Loday and Ronco give the free dendriform algebra on one generator as an algebra over binary planar trees \cite{Loday1998}. Thanks to dendriform algebras, Foissy proves \cite[proposition 31]{Foissy2002b} that the decorated Hopf algebra of  Loday and Ronco and  the decorated Hopf algebra of planar rooted trees are isomorphic.  Analogue theorems of the Cartier-Quillen-Milnor-Moore theorem have been proved: by Ronco \cite{Ronco2001} for dendriform algebras, by Chapoton \cite{Chapoton2002} for dendriform bialgebras and by Foissy \cite{Foissy2007} for bidendriform bialgebras. The bidendriform case implies that  $\FQSym$ is isomorphic to  one decorated Hopf algebra of planar rooted trees.

About quadri-algebras, Aguiar and Loday \cite{Aguiar2004} have determined a quadri-algebra structure on infinitesimal algebras and have focused on the free quadri-algebra on one generator.  Vallette \cite{Vallette2008} has proved some conjectures given by Aguiar and Loday in \cite[conjectures 4.2, 4.5 and 4.6]{Aguiar2004}. Foissy has presented the free quadri-algebra on one generator as a sub-object of $\FQSym$ \cite{Foissy2015}.

In this section, we recall the dendriform algebra structure and the quadri-algebra structure underlying the classical shuffle algebra. Then, we consider the case of weak shuffle algebras. We prove that they can be equipped with a dendriform structure yet only two weak shuffle products can be considered as coming from a quadri-algebra.

\subsection{Dendriform algebras}
\subsubsection{Background} 
\begin{defi}
	A dendriform algebra is a vector space  $\mathcal{D}$  equipped with two $\prec$ products $\succ$ such that $\forall x,y,z\in\mathcal{D}$,
	\begin{align*}
	(x\prec y)\prec z&=x\prec(y\prec z)+x\prec(y\succ z),\\
	(x\succ y)\prec z&=x\succ(y\prec z),\\
	(x\prec y)\succ z+ (x\succ y)\succ z&=x\succ(y\succ z).
	\end{align*}
\end{defi}
\begin{theo}
	Let $X$ be a countable alphabet and $\shuffle$ be the classical shuffle product. We define $\prec$ and $\succ$ respectively by: 
	\[au\prec bv=a(u\shuffle bv), ~ au\succ bv=b(au\shuffle v),\]
	for any letters $a$ and $b$ and any words $u$ and $v$.
	Then $(\K\langle X\rangle,\prec,\succ)$ is a dendriform algebra and for any words $u$ and $v$
	\[u\shuffle v=u\prec v+u\succ v.\] 
\end{theo}

\begin{theo}
	Let $X$ be a countable alphabet and $\shuffle$ be the classical shuffle product. We define $\wedge$ and $\vee$ respectively by: 
	\[ua\wedge vb=(u\shuffle vb)a, ~ ua\vee vb=(ua\shuffle v)b,\]
	for any letters $a$ and $b$ and any words $u$ and $v$.
	Then $(\K\langle X\rangle,\wedge,\vee)$ is a dendriform algebra and for any words $u$ and $v$
	\[u\shuffle v=u\wedge v+u\vee v.\]  
\end{theo}

\subsubsection{Weak shuffle products}
\begin{theo}
Let $X$ be a countable alphabet and $\Box$ be a weak shuffle product  such that $f_1(a\ot a)\in\{0,1\}$ for any letter $a\in X$. We define the products $\prec$ and $\succ$ respectively by: 
\[au\prec bv=f_1(a\ot b)a(u\Box bv), ~ au\succ bv=f_2(a\ot b)b(au\Box v),\]
for any letters $a$ and $b$ and any words $u$ and $v$.
Then $(\K\langle X\rangle,\prec,\succ)$ is a dendriform algebra.  
\end{theo}

\begin{proof}
	Let $\Box$ be a weak shuffle product and let $a$, $b$ and $c$ be three letters of $X$. Then:
	\begin{align*}
	(a\prec b)\prec c =& f_1(a\ot b)f_1(a\ot c)f_1(b\ot c)abc+f_1(a\ot b)f_1(a\ot c)f_2(b\ot c)acb,\\
	a\prec(b\Box c)=&  f_1(a\ot b)f_1(b\ot c)abc+f_1(a\ot c)f_2(b\ot c)acb,\\
	(a\succ b)\prec c=& f_2(a\ot b)f_1(b\ot c)f_1(a\ot c)bac+f_2(a\ot b)f_1(b\ot c)f_2(a\ot c)bca,\\
	a\succ (b\prec c)=& f_2(a\ot b)f_1(b\ot c)f_1(a\ot c)bac+f_2(a\ot b)f_1(b\ot c)f_2(a\ot c)bca,\\
	(a\Box b)\succ c=& f_1(a\ot b)f_2(a\ot c)cab+f_2(a\ot b)f_2(b\ot c)cba,\\
	a\succ (b\succ c)=& f_2(b\ot c)f_2(a\ot c)f_1(a\ot b)cab+f_2(b\ot c)f_2(a\ot c)f_2(a\ot b)cba.
	\end{align*}

Then $(a\succ b)\prec c=a\succ (b\prec c)$.
If the three letters are all distinct or only two of them are equal or $a=b=c$ with $f_1(a\ot a)=f_2(a\ot a)\in\{0,1\}$ the relations given by Theorem \ref{carctwshuffle} imply $(a\prec b)\prec c=a\prec(b\Box c)$ and $(a\Box b)\succ c=a\succ (b\succ c)$. If $a=b=c$ with $f_1(a\ot a)=1-f_2(a\ot a)$ then $(a\prec a)\prec a=a\prec (a\Box a)$ and $(a\Box a)\succ a=a\succ (a\succ a)$ if and only if $f_1(a\ot a)\in\{0,1\}$ and then $f_1(a\ot a)f_2(a\ot a)=0$.

We assume now there exists an integer $n\leq 3$ such that, for any non-empty words $u$, $v$ and $w$ with $\length(u)+\length(v)+\length(w)=n$, relations $(u\prec v)\prec w =u\prec(v\Box w)$, $(u\succ v)\prec w=u\succ (v\prec w)$ and $(u\Box v)\succ w=u\succ (v\succ w)$ are satisfied.

Let $u$, $v$ and $w$ be three non-empty words such that $\length(u)+\length(v)+\length(w)=n+1$. There exist three letters $a$, $b$ and $c$, not necessarily distinct and three words $u_1$, $v_1$ and $w_1$, not necessarily non-empty, such that $u=au_1$, $v=bv_1$ and $w=cw_1$. Then 
\begin{enumerate}
	\item \begin{align*}
	(u\prec v)\prec w=& f_1(a\ot b)f_1(a\ot c)a\big[(u_1\Box bv_1)\Box cw_1\big]=f_1(a\ot b)f_1(a\ot c)a\big[u_1\Box (bv_1\Box cw_1)\big]\\
	=&f_1(a\ot b)f_1(a\ot c)f_1(b\ot c)a\big[u_1\Box b(v_1\Box cw_1)\big]\\
	&+f_1(a\ot b)f_1(a\ot c)f_2(b\ot c)a\big[u_1\Box c(bv_1\Box w_1)\big],\\
	u\prec (v\Box w)=& f_1(b\ot c)f_1(a\ot b)a\big[u_1\Box b(v_1\Box cw_1)\big]\\
	&+f_2(b\ot c)f_1(a\ot c)f_1(a\ot c)a\big[u_1\Box c(bv_1\Box w_1)\big].
	\end{align*}
	\item \begin{align*}
	(u\succ v)\prec w=& f_2(a\ot b)f_1(b\ot c)b\big[(au_1\Box v_1)\Box cw_1\big],\\
	u\succ (v\prec w)=& f_1(b\ot c)f_2(a\ot b)b\big[au_1\Box (v_1\Box cw_1)\big].
	\end{align*}
	\item \begin{align*}
	(u\Box v)\succ w=& f_1(a\ot b)f_2(a\ot c)c\big[a(u_1\Box bv_1)\Box cw_1\big]+f_2(a\ot b)f_2(b\ot c)c\big[b(au_1\Box v_1)\Box cw_1\big],\\
	u\succ (v\succ w)=& f_2(b\ot c)f_2(a\ot c)c\big[au_1\Box (bv_1\Box w_1)\big]=f_2(b\ot c)f_2(a\ot c)c\big[(au_1\Box bv_1)\Box w_1\big]\\
	=&f_2(b\ot c)f_2(a\ot c)f_1(a\ot b)c\big[a(u_1\Box bv_1)\Box w_1\big]\\
	&+f_2(b\ot c)f_2(a\ot c)f_2(a\ot b)c\big[b(au_1\Box v_1)\Box w_1\big].
	\end{align*}
\end{enumerate}

Thus, $(u\prec v)\prec w =u\prec(v\Box w)$, $(u\succ v)\prec w=u\succ (v\prec w)$ and $(u\Box v)\succ w=u\succ (v\succ w)$.
\end{proof}

By considering the right hand side rather than the left hand side, we get the following definition and theorem.
\begin{defi}
	Let $X$ be a countable alphabet. An end weak shuffle product on $\K\langle X\rangle$ is an associative and commutative product $\Box_E$ such that for any $(a,b)\in(X)^2$ and any $(u,v)\in(X^*)^2$ then
	\[ua\Box_E vb=f_{1,E}(a\ot b)(u\Box_E vb)a+f_{2,E}(a\ot b)(ua\Box_E v)b,\]
	where $f_{1,E}$ and $f_{2,E}$ are linear maps from $\K.X\ot \K.X$ to $\K$,  $u\Box_E 0=0\Box_E u=0$ and $u\Box_E 1=1\Box_E u=u$.
\end{defi}
\begin{theo}
	Let $X$ be a countable alphabet and let $\Box_E$ be an end weak shuffle product  such that $f_{1,E}(a\ot a)\in\{0,1\}$ for any letter $a\in X$. We define the products $\wedge$ and $\vee$ by: 
	\[ua\wedge vb=f_{1,E}(a\ot b)(u\Box_E vb)a, ~ au\vee bv=f_{2,E}(a\ot b)(ua\Box_E v)b,\]
	for any letters $a$ and $b$ and any words $u$ and $v$.
	Then $(\K\langle X\rangle,\wedge,\vee)$ is a dendriform algebra.  
\end{theo}

\Rq{1} Let $\alpha$ be a real number. Let $\Box$ be the weak shuffle product satisfying $f_1(a\ot a)=1-f_2(a\ot a)=\alpha$ for a unique letter $a$. Even if $\Box$ does not depend on the value of $\alpha$, to express the algebra as a dendriform algebra the assumption $\alpha\in\{0,1\}$ is necessary.

\subsection{Quadri-algebras}
\subsubsection{Background}
\begin{defi}
	A quadri-algebra is $\mathcal{Q}$ is a vector space equipped with four products $\searrow$, $\nearrow$, $\nwarrow$ and $\swarrow$ such that: for any $x,y,z\in\mathcal{Q}$,
	\begin{align*}
	(x\nwarrow y)\nwarrow z&=x\nwarrow(y\cdot z), & (x\nearrow y)\nwarrow z&=x\nearrow(y\prec z),\\ 
	(x\swarrow y)\nwarrow z&=x\swarrow(y\wedge z), & (x\searrow y)\nwarrow z&=x\searrow(y\nwarrow z),\\ 
	(x\prec y)\swarrow z&=x\swarrow(y\vee z), & (x\succ y)\swarrow z&=x\searrow(y\swarrow z),\\ 
	\end{align*}
	and
	\begin{align*}
	(x\wedge y)\nearrow z&=x\nearrow(y\succ z),\\
	(x\vee y)\nearrow z&=x\searrow(y\nearrow z),\\
	(x\cdot y)\searrow z&=x\searrow(y\searrow z).
	\end{align*}
	where  \begin{align*}
	x\prec y&=x\nwarrow y+x\swarrow y,& x\wedge y&=x\nearrow y+x\nwarrow y,\\
	x\succ y&=x\nearrow y+x\searrow y,& x\vee y&=x\searrow y+x\swarrow y,\\
	\end{align*}
	and
	\[x\cdot y= x\nwarrow y+x\swarrow y+x\nearrow y+x\searrow y=x\prec y+x\succ y= x\wedge y+x\vee y.\]
\end{defi}

\begin{theo}\label{Quadri}
	Let $X$ be a countable alphabet and let $\shuffle$ be the classical shuffle product. The products $\searrow$, $\nearrow$, $\nwarrow$ and $\swarrow$ are defined as follow: 
	\[auc\nwarrow bvd=a(u\shuffle bvd)c, ~ auc\swarrow bvd=a(uc\shuffle bv)d,\] 
	\[auc\nearrow bvd=b(au\shuffle vd)c, ~ auc\searrow bvd=b(auc\shuffle v)d\]
	for any letters $a$, $b$, $c$ and $d$ and any words $u$ and $v$.
	Then $(\K\langle X\rangle,\searrow, \nearrow, \nwarrow,\swarrow)$ is a quadri-algebra. 
\end{theo}
\begin{proof} It is proved in \cite[Section 1.8]{Aguiar2004}. The main ingredient of the proof is the following statement:
	for any letters $a$, $b$, $c$ and $d$ and any words $u$ and $v$ we have
\[auc\shuffle bvd=a(uc\shuffle bvd)+b(auc\shuffle vd)=(au\shuffle bvd)c+(auc\shuffle bv)d.\]
\end{proof}

\subsubsection{Weak shuffle algebras}
\begin{prop}
	Let $X$ be a countable alphabet of cardinality at least 2. Let $\Box$ be a weak shuffle product. There exists an end weak shuffle product $\Box_E$ such that $\Box=\Box_E$ if, and only if, $\Box$ is the null product or the classical shuffle product.
\end{prop}
\begin{proof}
	It is sufficient to prove the proposition for an alphabet of cardinality 2 and assume images of functions $f_1$, $f_2$, $f_{1,E}$ and $f_{2,E}$ are subsets of $\{0,1\}$. Let $C$ be the $6$-tuple $C=\Bigg( f_1(a\ot b),f_1(b\ot a), f_1(a\ot a), f_2(a\ot a), f_1(b\ot b), f_2(b\ot b)\Bigg)$.
	\begin{description}
		\item[Case $C=(1,0,0,0,0,0)$.] If $\Box=\Box_E$ then 
		\begin{align*}
		a\Box_E ba=&\left(f_{1,E}(a\ot a)+f_{2,E}(a\ot a)f_{1,E}(a\ot b)\right)baa + f_{2,E}(a\ot a)f_{1,E}(b\ot a)aba\\
				  =&a\Box ba=aba.
		\end{align*}
		Thus $f_{2,E}(a\ot a)=1$ and then $a\Box_E a=\left(f_{1,E}(a\ot a)+1\right)aa\neq0$ and yet $a\Box a=0$. Contradiction.
		
		\item[Cases $C=(1,0,1,1,0,0)$ and $C=(1,0,1,0,0,0)$.] We recall that these two cases are isomorphic. If $\Box=\Box_E$ then 
		\begin{align*}
		a\Box_E ba=&\left(f_{1,E}(a\ot a)+f_{2,E}(a\ot a)f_{1,E}(a\ot b)\right)baa + f_{2,E}(a\ot a)f_{1,E}(b\ot a)aba\\
		=&\left(f_{1,E}(a\ot a)f_{1,E}(a\ot b)+f_{2,E}(a\ot a)\right)baa + f_{1,E}(a\ot a)f_{1,E}(b\ot a)aba\\
		=&ba\Box_E a=a\Box ba=aba.
		\end{align*}
		Thus $f_{1,E}(a\ot a)=f_{2,E}(a\ot a)=f_{1,E}(b\ot a)=1$ and $f_{1,E}(a\ot b)=-1$. Contradiction.
		
		\item[Cases $C=(1,0,1,0,1,1)$ and $C=(1,0,1,1,1,1)$.] The same calculations as in the previous case answer the question.
		
		\item[Case $C=(1,0,0,0,1,1)$.] If $\Box=\Box_E$ then 
		\begin{align*}
		ba\Box_E b=&f_{1,E}(a\ot b)\left(f_{1,E}(b\ot b)+f_{2,E}(b\ot b)\right)bba + f_{1,E}(b\ot a)bab\\
		=&ba\Box b=bba+bab.
		\end{align*}
		Thus $f_{1,E}(a\ot b)=f_{1,E}(b\ot a)=f_{1,E}(a\ot a)=f_{2,E}(a\ot a)=f_{1,E}(b\ot b)=f_{2,E}(b\ot b)=1$ with $f_{1,E}(b\ot b)+f_{2,E}(b\ot b)=1$. Contradiction.
	
		\item[Cases $C=(0,0,1,1,0,0)$.] If $\Box=\Box_E$ then 
		\begin{align*}
		ab\Box_E a=&f_{1,E}(b\ot a)\left(f_{1,E}(a\ot a)+f_{2,E}(a\ot a)\right)aab + f_{1,E}(a\ot b)aba\\
		=&ab\Box a=aab.
		\end{align*}
		Thus $f_{1,E}(a\ot b)=0$, $f_{1,E}(b\ot a)=1$ and $f_{1,E}(a\ot a)+f_{2,E}(a\ot a)=1$. Contradiction.
		
		\item[Cases $C=(0,0,1,1,1,1)$.] The same calculations as in the previous case answer the question.
	\end{description}
\end{proof}
\begin{cor}	The construction used in Theorem \ref{Quadri} does not lead to a quadri-algebra structure on a weak shuffle product $\Box$ except if $\Box$ is the null shuffle or the classical shuffle.
\end{cor}

\section{Relations on weak stuffle products}
\begin{prop}\label{relationsstuffle}
	Let $X$ be a countable alphabet, let $a$, $b$ and $c$ be three distinct letters in $X$ and $\Box$ a weak stuffle product. Then:
	\begin{enumerate}
		\item By using the maps $f_1$ and $f_2$ coming from $\Box$, we define the product $\Box^{'}$ by:
		$au\Box^{'}bv=f_1(a\ot b)a(u\Box^{'}bv)+f_2(a\ot b)b(au\Box^{'}v)$ for any letters $a$ and $b$ and any words $u$ and $v$.
		The product $\Box^{'}$ is a weak shuffle product.
		\item The function $f_3$ is associative and commutative.
 		\item If $f_3(a\ot a)\neq 0$ then $f_1(a\ot a)=f_2(a\ot a)\in\{0,1\}$.
		\item If $f_3(a\ot b)\neq 0$ then $f_1(a\ot a)=f_2(a\ot a)\in\{0,1\}$ and $f_1(b\ot b)=f_2(b\ot b)\in\{0,1\}$.
		\item If $f_3(a\ot a)\in\K^*a$ then $f_1(b\ot a)\in\{0,1\}$.
		\item If $f_3(a\ot a)\in\K^*b$ then 
			\begin{enumerate}
				\item If $f_3(a\ot b)\neq 0$ or $f_3(b\ot b)\neq 0$ or there exists $x\in X\setminus\{a,b\}$ such that $f_3(b\ot x)\neq 0$  then $f_1(a\ot a)=f_2(a\ot a)=f_1(b\ot b)=f_2(b\ot b)=f_1(a\ot b)=f_1(b\ot a)\in\{0,1\}$.
				\item If $f_3(a\ot b)=0$ and $f_3(b\ot b)=0$ then 
					\begin{enumerate}
						\item either $f_1(a\ot a)=f_2(a\ot a)=f_1(b\ot b)=f_2(b\ot b)=f_1(a\ot b)=f_1(b\ot a)\in\{0,1\}$,
						\item or $f_1(a\ot a)=f_2(a\ot a)=f_1(b\ot a)=1$, $f_1(b\ot b)+f_2(b\ot b)=1$ and $f_1(a\ot b)=0$.
					\end{enumerate}
				\item For any $x\in X\setminus\{a,b\}$ then 
					\begin{enumerate}
						\item $f_1(a\ot x)=f_1(b\ot x)$,
						\item $f_1^2(x\ot a)=f_1(x\ot b)$.
					\end{enumerate}
			\end{enumerate}
		\item If $f_3(a\ot b)\in\K^*a$ then:
			\begin{enumerate}
				\item $f_1(b\ot a)=f_1(a\ot a)f_1(a\ot b)=f_1(b\ot a)f_1(b\ot b)$.
				\item $f_1(a\ot b)=f_1(b\ot b)$.
				\item For any $x\in X\setminus\{a,b\}$ such that $f_3(b\ot x)\notin\K^* x$ then 
					\begin{enumerate}
						\item $f_1(a\ot x)=f_1(b\ot x)$,
						\item $f_1(x\ot a)\left[1-f_1(x\ot b)\right]=0$.
					\end{enumerate} 
				\item For any $x\in X\setminus\{a,b\}$ such that $f_3(b\ot x)\in\K^* x$ then 
				\begin{enumerate}
					\item $f_1(b\ot a)=f_1(x\ot a)f_1(x\ot b)$,
					\item $f_1(b\ot x)=f_1(a\ot b)f_1(a\ot x)$,
				\end{enumerate} 
			\end{enumerate}
		\item If $f_3(a\ot b)\in\K^*c$ then:
		\begin{enumerate}
			\item $f_1(c\ot c)=f_2(c\ot c)\in\{0,1\}$.
			\item $f_1(b\ot a)=f_1(c\ot a)=f_1(a\ot a)$.
			\item $f_1(a\ot b)=f_1(c\ot b)=f_1(b\ot b)$.
			\item $f_1(a\ot c)=f_1(a\ot a)f_1(b\ot b)=f_1(b\ot c)=f_1(c\ot c)$.
		\end{enumerate}
	\end{enumerate}
\end{prop}

\begin{proof}
	\begin{enumerate}
		\item Let $a$ and $b$ be two letters and let $u$ and $v$ be two words. By using words of length $\length(u)+\length(v)+2$ appearing in $au\Box bv$ , we get the statement. In the sequel, the use of the relations given in Theorem \ref{carctwshuffle} is implied.
		\item By using words of length $1$ appearing in $x\Box y$, $x\Box y$, $(x\Box y)\Box z$ and $x\Box (y\Box z)$ for any letters $x$, $y$, $z$, we prove that the function $f_3$ is associative and commutative.
		\item We assume $f_3(a\ot a)\neq 0$. Since $a\Box aa=aa\Box a$ and $(a\Box a)\Box aa=a\Box(a\Box aa)$ then $f_1(a\ot a)=f_2(a\ot a)\in\{0,1\}$.
		\item We assume $f_3(a\ot b)\neq 0$. Since $a\Box ab=ab\Box a$, $b\Box ba=ba\Box b$, $(a\Box b)\Box a=(a\Box a)\Box b$ and $(b\Box a)\Box b=(b\Box b)\Box a$ then $f_1(a\ot a)=f_2(a\ot a)\in\{0,1\}$ and $f_1(b\ot b)=f_2(b\ot b)\in\{0,1\}$.
		\item This item is proved by using $(a\Box a)\Box b=(a\Box b)\Box a$ and $a\Box(a\Box ba)=(a\Box a)\Box ba$.
		\item We assume $f_3(a\ot a)\in\K^*b$. 
			\begin{enumerate}
				\item If  $f_3(a\ot b)\neq 0$ or $f_3(b\ot b)\neq 0$, since $f_1(a\ot a)=f_2(a\ot a)\in\{0,1\}$, $f_1(b\ot b)=f_2(b\ot b)\in\{0,1\}$, $(a\Box b)\Box a=(a\Box a)\Box b$ and $(a\Box a)\Box aa=a\Box(a\Box aa)$, then $f_1(a\ot a)=f_2(a\ot a)=f_1(b\ot b)=f_2(b\ot b)=f_1(a\ot b)=f_1(b\ot a)\in\{0,1\}$.
				\item If $f_3(a\ot b)=0$ and $f_3(b\ot b)= 0$, since $f_1(a\ot a)=f_2(a\ot a)\in\{0,1\}$, $(a\Box b)\Box a=(a\Box a)\Box b$ and $(a\Box a)\Box aa=a\Box(a\Box aa)$ then we prove the relations.
				\item This item is proved thanks to the relation $(a\Box b)\Box a=(a\Box a)\Box b$.
			\end{enumerate}
		\item We assume $f_3(a\ot b)\in\K^*a$.
			\begin{enumerate}
				\item This item is proved by using $f_1(a\ot a)=f_2(a\ot a)\in\{0,1\}$, $f_1(b\ot b)=f_2(b\ot b)\in\{0,1\}$, $(a\Box b)\Box a=(a\Box a)\Box b$ and $(b\Box a)\Box b=(b\Box b)\Box a$.
				\item By using $(b\Box b)\Box a=(b\Box a)\Box b$ and $(a\Box b)\Box ba=a\Box (b\Box ba)$ we prove $f_1(a\ot b)=f_1(b\ot b)$.
				\item Those two subitems are proven by using $(a\Box b)\Box x=(a\Box x)\Box b=(b\Box x)\Box a$.
				\item Those two subitems are proven by using $(a\Box b)\Box x=(a\Box x)\Box b=(b\Box x)\Box a$.
			\end{enumerate}
		\item We assume $f_3(a\ot b)\in\K^*c$. Then $f_1(a\ot a)=f_2(a\ot a)\in\{0,1\}$ and $f_1(b\ot b)=f_2(b\ot b)\in\{0,1\}$. By using the relations $(a\Box b)\Box c=(a\Box c)\Box b=(b\Box c)\Box a$, $(a\Box b)\Box b=(b\Box b)\Box a$, $(b\Box a)\Box a=(a\Box a)\Box b$, $(a\Box b)\Box aa=a\Box(b\Box aa)=b\Box (a\Box aa$ and $(b\Box a)\Box bb=b\Box (a\Box bb)=a\Box (b\Box bb)$ we prove all subitems.
	\end{enumerate}
\end{proof}

\Ex{}
\begin{enumerate}
	\item The $q$-shuffle product associated to the Schlesinger-Zudilin model is the weak stuffle product where $f_1(y\ot p )=f_1(y\ot y)=f_1(p\ot p)=f_2(p\ot p)=1$, $f_1(p\ot y)=f_2(y\ot y)=0$, $f_3(p\ot p)=p$, $f_3(y\ot p)=f_3(y\ot y)=0$. 
	\item The $q$-shuffle product associated to the Bradley-Zhao model is the weak stuffle product where $f_1(y\ot p )=f_1(y\ot \overline{p})=f_1(p\ot \overline{p})=f_1(\overline{p}\ot p)=f_1(p\ot p)=f_2(p\ot p)=f_1(\overline{p}\ot \overline{p})=f_2(\overline{p}\ot \overline{p})=f_1(y\ot y)=1$, $f_1(p\ot y)=f_1(\overline{p}\ot y)=f_2(y\ot y)=0$, $f_3(p\ot p)=p$, $f_3(\overline{p}\ot \overline{p})=-\overline{p}$ $f_3(y\ot p)=f_3(y\ot y)=f_3(y\ot \overline{p})=f_3(p\ot \overline{p})=0$. 
\end{enumerate}

\begin{cor}
	Let $X=\{x_1,\dots,x_n\dots\}$ be an infinite countable alphabet. We assume $\Box$ is a weak stuffle product such that $f_3(x_i\ot x_j)\in\K^*x_{i+j}$ for any positive integers $i$ and $j$. Then, the underlying weak shuffle produit is either the null shuffle product or the classical stuffle product \emph{i.e.} ($f_1\equiv 0$ and $f_2\equiv 0$) or ($f_1(a\otimes b )=1$ and  $f_2(a\otimes b)=1$ for any letters $a$ and $b$).
\end{cor}

\begin{proof}
	We use an inductive proof. First of all, since $f_3(x_i\ot x_i)\neq 0$ for any positive integer $i$, we have $f_1(x_i\ot x_i)=f_2(x_i\ot x_i)$. Besides, $f_3(x_1\ot x_1)=x_2\neq x_1$ and $f_3(x_2\ot x_2)\neq 0$, so $f_1(x_1\ot x_1)=f_2(x_1\ot x_1)=f_1(x_2\ot x_2)=f_2(x_2\ot x_2)=f_1(x_1\ot x_2)=f_1(x_2\ot x_1)\in\{0,1\}$.

	We assume there exists $n\in\N^*$ such that $n\geq 2$ and $f_1(x_1\ot x_1)=f_1(x_1\ot x_m)$ for any $m\in\llbracket 1,n \rrbracket$. Then, $f_3(x_1\ot x_n)=x_{n+1}$ and $f_1(x_1\ot x_{n+1})=f_1(x_1\ot x_1)f_1(x_1\ot x_n)=f_1(x_1\ot x_1)$. Thus, $f_1(x_1\ot x_1)=f_1(x_1\ot x_n)$ for any positive integer $n$.
	
	We assume now there exists $k\in\N^*$ such that $f_1(x_1\ot x_1)=f_1(x_i\ot x_j)$ for any $i\in\llbracket 1, k\rrbracket$ and any positive integer $j$. For any $i\in\llbracket 1, k\rrbracket$,  we know $f_3(x_i\ot x_{k+1-i})=x_{k+1}$ so, $f_1(x_{k+1}\ot x_{i})=f_1(x_{k+1-i}\ot x_{i})=f_1(x_1\ot x_1)$. Besides, we know 
	\[f_1(x_{k+1}\ot x_{k+1})=f_2(x_{k+1}\ot x_{k+1})=f_1(x_1\ot x_{k+1})=f_1(x_1\ot x_1).\]
	Since $f_3(x_{k+1}\ot x_1)=x_{k+2}$, we have $f_1(x_{k+1}\ot x_{k+2})=f_1(x_{1}\ot x_{k+2})=f_1(x_1\ot x_1)$. We assume there exists a positive integer $j$ such that $f_1(x_{k+1}\ot x_{k+1+p})=f_1(x_1\ot x_1)$ for any $p\in\llbracket 1, j\rrbracket$. As $f_3(x_{k+1}\ot x_{j+1})=x_{k+j+2}$ then \[f_1(x_{k+1}\ot x_{k+j+2})=f_1(x_{k+1}\ot x_{k+1})f_1(x_{k+1}\ot x_{j+1})=f_1(x_1\ot x_1).\] 
	Finally, ($f_1\equiv 0$ and $f_2\equiv 0$) or ($f_1(a\otimes b )=1$ and  $f_2(a\otimes b)=1$ for any letters $a$ and $b$).
\end{proof}

By using the commutativity and the associativity of $k_3$ we have:
\begin{lemma}
	Let $X=\{a,b\}$ be an alphabet of cardinality 2 and let $\Box$ be a weak stuffle product. The map $f_3$ is one of the following:
	\begin{enumerate}
		\item\label{ita} There exists $(\lambda,\mu)\in(\K^*)^2$ such that $f_3(a\ot a)=\lambda b$, $f_3(a\ot b)=\mu a$ and $f_3(b\ot b)=\mu b$.
		\item\label{itb} There exists $(\lambda,\mu)\in(\K^*)^2$ such that $f_3(a\ot a)=\lambda a$, $f_3(a\ot b)=\mu a$ and $f_3(b\ot b)=\frac{\mu^2}{\lambda}a$.
		\item\label{itc} There exists $(\lambda,\mu)\in(\K^*)^2$ such that $f_3(a\ot a)=\lambda a$, $f_3(a\ot b)=\mu a$ and $f_3(b\ot b)=\mu b$.
		\item\label{itd} There exists $(\lambda,\mu)\in(\K^*)^2$ such that $f_3(a\ot a)=0$, $f_3(a\ot b)=\mu a$ and $f_3(b\ot b)=\lambda b$.
		\item\label{ite} There exists $(\lambda,\mu)\in(\K^*)^2$ such that $f_3(a\ot a)=\lambda a$, $f_3(a\ot b)=0$ and $f_3(b\ot b)=\mu b$.
		\item\label{itf} There exists $\lambda\in\K^*$ such that $f_3(a\ot a)=\lambda b$, $f_3(a\ot b)=0$ and $f_3(b\ot b)=0$.
		\item\label{itg} There exists $\lambda\in\K^*$ such that $f_3(a\ot a)=\lambda a$, $f_3(a\ot b)=0$ and $f_3(b\ot b)=0$.
		\item\label{ith} The map $f_3$ is the null map.
	\end{enumerate}
\end{lemma}
By using Proposition \ref{relationsstuffle} we have:
\begin{prop}
	Let $X=\{a,b\}$ be an alphabet of cardinality 2 and let $\Box$ be a weak stuffle product. In the previous lemma, if $f_3$ satisfies
	\begin{enumerate}
		\item Item \ref{ita} or item \ref{itb}, then there are two cases:
		\begin{itemize}
			\item $f_1(a\otimes b )=1$ and  $f_2(a\otimes b)=1$ for any  $(a,b)\in X^2$,
			\item $f_1\equiv0$ and $f_2\equiv0$.
		\end{itemize}
		\item Item \ref{itc} or item \ref{itd}, then there are four cases:
		\begin{itemize}
			\item $f_1(a\otimes b )=1$ and  $f_2(a\otimes b)=1$ for any  $(a,b)\in X^2$,
			\item $f_1\equiv0$ and $f_2\equiv0$,
			\item $f_1(b\ot a)=f_1(a\ot b)=f_1(b\ot b)=f_2(b\ot b)=0$ and $f_1(a\ot a)=f_2(a\ot a)=1$,
			\item $f_1(a\ot b)=f_1(b\ot b)=f_2(b\ot b)=1$ and $f_1(b\ot a)=f_1(a\ot a)=f_2(a\ot a)=0$.
		\end{itemize}
		\item Item \ref{ite}, then we have:
		\begin{itemize}
			\item $f_1(a\ot b)\in\{0,1\}$,
			\item $f_1(b\ot a)\in\{0,1\}$,
			\item $f_1(a\ot a)=f_2(a\ot a)\in\{0,1\}$,
			\item $f_1(b\ot b)=f_2(b\ot b)\in\{0,1\}$.
		\end{itemize}
		\item Item \ref{itf}, then there are three cases:
		\begin{itemize}
			\item $f_1(a\otimes b )=1$ and  $f_2(a\otimes b)=1$ for any  $(a,b)\in X^2$,
			\item $f_1\equiv0$ and $f_2\equiv0$,
			\item $f_1(a\ot a)=f_2(a\ot a)=f_1(b\ot a)=1$, $f_1(a\ot b)=0$ and $f_1(b\ot b)+f_2(b\ot b)=1$
		\end{itemize}
		\item Item \ref{itg}, then we have:
		\begin{itemize}
			\item $f_1(b\ot a)\in\{0,1\}$,
			\item $f_1(a\ot a)=f_2(a\ot a)\in\{0,1\}$.
		\end{itemize}
		\item Item \ref{ith}, then we give the answer in Theorem \ref{carctwshuffle}.
	\end{enumerate}
\end{prop}

\begin{lemma}
	Let $X=\{a,b,c\}$ be an alphabet of cardinality 3 and let $\Box$ be a weak stuffle product. The map $f_3$ is one of the following:
	\begin{enumerate}
		\item \label{un} There exists $(\lambda,\gamma,\mu)\in(\K^*)^3$ such that $\gamma\mu=\lambda^2$, $f_3(a\ot b)=\lambda c$, $f_3(a\ot c)=\lambda a$, $f_3(b\ot c)=\lambda b$, $f_3(a\ot a)=\gamma b$, $f_3(b\ot b)=\mu a$ and $f_3(c\ot c)=\lambda c$.
		
		\item \label{deux} There exists $(\lambda,\gamma,\mu)\in(\K^*)^3$ such that $f_3(a\ot b)=\lambda c$, $f_3(a\ot c)=\gamma c$, $f_3(b\ot c)=\mu c$, $f_3(a\ot a)=\gamma a$, $f_3(b\ot b)=\frac{\lambda\mu}{\gamma}a$ and $f_3(c\ot c)=\frac{\gamma\mu}{\lambda} c$.

		\item There exists $(\lambda,\gamma,\mu)\in(\K^*)^3$ such that $f_3(a\ot b)=\lambda c$, $f_3(a\ot c)=\gamma c$, $f_3(b\ot c)=\mu c$, $f_3(a\ot a)=\gamma a$, $f_3(b\ot b)=\mu b$ and $f_3(c\ot c)=\frac{\gamma\mu}{\lambda} c$.
		
		\item There exists $(\lambda,\gamma,\mu)\in(\K^*)^3$ such that $f_3(a\ot b)=\lambda c$, $f_3(a\ot c)=\gamma c$, $f_3(b\ot c)=\mu c$, $f_3(a\ot a)=\gamma a$, $f_3(b\ot b)=\frac{\lambda\mu}{\gamma}c$ and $f_3(c\ot c)=\frac{\gamma\mu}{\lambda} c$.
		
		\item \label{cinq} There exists $(\lambda,\gamma,\mu)\in(\K^*)^3$ such that $f_3(a\ot b)=\lambda c$, $f_3(a\ot c)=\gamma c$, $f_3(b\ot c)=\mu c$, $f_3(a\ot a)=\frac{\gamma^2}{\mu} b$, $f_3(b\ot b)=\frac{\lambda\mu}{\gamma}c$ and $f_3(c\ot c)=\frac{\gamma\mu}{\lambda} c$.
		
		\item \label{six} There exists $(\lambda,\gamma,\mu)\in(\K^*)^3$ such that $f_3(a\ot b)=\lambda c$, $f_3(a\ot c)=\gamma c$, $f_3(b\ot c)=\mu c$, $f_3(a\ot a)=\frac{\lambda\gamma}{\mu} c$, $f_3(b\ot b)=\frac{\mu^2}{\gamma}a$ and $f_3(c\ot c)=\frac{\gamma\mu}{\lambda} c$.
		
		\item There exists $(\lambda,\gamma,\mu)\in(\K^*)^3$ such that $f_3(a\ot b)=\lambda c$, $f_3(a\ot c)=\gamma c$, $f_3(b\ot c)=\mu c$, $f_3(a\ot a)=\frac{\lambda\gamma}{\mu} c$, $f_3(b\ot b)=\mu b$ and $f_3(c\ot c)=\frac{\gamma\mu}{\lambda} c$.
		
		\item \label{huit} There exists $(\lambda,\gamma,\mu)\in(\K^*)^3$ such that $f_3(a\ot b)=\lambda c$, $f_3(a\ot c)=\gamma c$, $f_3(b\ot c)=\mu c$, $f_3(a\ot a)=\frac{\lambda\gamma}{\mu} c$, $f_3(b\ot b)=\frac{\lambda\mu}{\gamma}c$ and $f_3(c\ot c)=\frac{\gamma\mu}{\lambda} c$.
				
		\item There exists $(\lambda,\gamma)\in(\K^*)^2$ such that $f_3(a\ot b)=\lambda c$, $f_3(a\ot c)=\gamma c$, $f_3(b\ot c)=0$, $f_3(a\ot a)=\gamma a$, $f_3(b\ot b)=0$ and $f_3(c\ot c)=0$.
		
		\item There exists $\lambda\in\K^*$ such that $f_3(a\ot b)=\lambda b$, $f_3(a\ot c)=\lambda c$, $f_3(b\ot c)=0$, $f_3(a\ot a)=\lambda a$, $f_3(b\ot b)=0$ and $f_3(c\ot c)=0$.
		
		\item There exists $(\lambda,\gamma)\in(\K^*)^2$ such that $f_3(a\ot b)=\lambda b$, $f_3(a\ot c)=\lambda c$, $f_3(b\ot c)=0$, $f_3(a\ot a)=\lambda a$, $f_3(b\ot b)=0$ and $f_3(c\ot c)=\gamma b$.
		
		\item There exists $(\lambda,\gamma)\in(\K^*)^2$ such that $f_3(a\ot b)=\lambda b$, $f_3(a\ot c)=\lambda c$, $f_3(b\ot c)=0$, $f_3(a\ot a)=\lambda a$, $f_3(b\ot b)=0$ and $f_3(c\ot c)=\gamma c$.
			
		\item There exists $(\lambda,\gamma,\mu)\in(\K^*)^3$ such that $f_3(a\ot b)=\lambda b$, $f_3(a\ot c)=\lambda c$, $f_3(b\ot c)=0$, $f_3(a\ot a)=\lambda a$, $f_3(b\ot b)=\gamma b$ and $f_3(c\ot c)=\mu c$.
		
		\item \label{quatorze} There exists $(\lambda,\gamma,\mu)\in(\K^*)^3$ such that $f_3(a\ot b)=\lambda c$, $f_3(a\ot c)=0$, $f_3(b\ot c)=0$, $f_3(a\ot a)=\gamma c$, $f_3(b\ot b)=\mu c$ and $f_3(c\ot c)=0$.
		
		\item \label{quinze} There exists $(\lambda,\gamma)\in(\K^*)^2$ such that $f_3(a\ot b)=\lambda c$, $f_3(a\ot c)=0$, $f_3(b\ot c)=0$, $f_3(a\ot a)=\gamma b$, $f_3(b\ot b)=0$ and $f_3(c\ot c)=0$.
		
		\item There exists $\lambda\in\K^*$ such that $f_3(a\ot b)=\lambda c$, $f_3(a\ot c)=0$, $f_3(b\ot c)=0$, $f_3(a\ot a)=0$, $f_3(b\ot b)=0$ and $f_3(c\ot c)=0$.
		
		\item There exists $(\lambda,\gamma,\mu,\tau)\in(\K^*)^4$ such that $\gamma\mu=\lambda^2$, $f_3(a\ot b)=\lambda a$, $f_3(a\ot c)=0$, $f_3(b\ot c)=0$, $f_3(a\ot a)=\gamma a$, $f_3(b\ot b)=\mu a$ and $f_3(c\ot c)=\tau c$.
		
		\item There exists $(\lambda,\gamma,\tau)\in(\K^*)^3$ such that $f_3(a\ot b)=\lambda a$, $f_3(a\ot c)=0$, $f_3(b\ot c)=0$, $f_3(a\ot a)=\gamma a$, $f_3(b\ot b)=\lambda b$ and $f_3(c\ot c)=\tau c$.
		
		\item There exists $(\lambda,\gamma,\tau)\in(\K^*)^3$ such that $f_3(a\ot b)=\lambda a$, $f_3(a\ot c)=0$, $f_3(b\ot c)=0$, $f_3(a\ot a)=\gamma b$, $f_3(b\ot b)=\lambda b$ and $f_3(c\ot c)=\tau c$.
		
		\item There exists $(\lambda,\gamma)\in(\K^*)^2$ such that $f_3(a\ot b)=\lambda a$, $f_3(a\ot c)=0$, $f_3(b\ot c)=0$, $f_3(a\ot a)=\gamma c$, $f_3(b\ot b)=\lambda b$ and $f_3(c\ot c)=0$.
		
		\item There exists $(\lambda,\tau)\in(\K^*)^2$ such that $f_3(a\ot b)=\lambda a$, $f_3(a\ot c)=0$, $f_3(b\ot c)=0$, $f_3(a\ot a)=0$, $f_3(b\ot b)=\lambda b$ and $f_3(c\ot c)=\tau c$.
		
		\item There exists $(\lambda,\gamma,\mu\in(\K^*)^3$ such that $\gamma\mu=\lambda^2$, $f_3(a\ot b)=\lambda a$, $f_3(a\ot c)=0$, $f_3(b\ot c)=0$, $f_3(a\ot a)=\gamma a$, $f_3(b\ot b)=\mu a$ and $f_3(c\ot c)=0$.
		
		\item There exists $(\lambda,\gamma,\tau)\in(\K^*)^3$ such that $f_3(a\ot b)=\lambda a$, $f_3(a\ot c)=0$, $f_3(b\ot c)=0$, $f_3(a\ot a)=\gamma a$, $f_3(b\ot b)=\lambda b$ and $f_3(c\ot c)=0$.

		\item There exists $(\lambda,\gamma)\in(\K^*)^2$ such that $f_3(a\ot b)=\lambda a$, $f_3(a\ot c)=0$, $f_3(b\ot c)=0$, $f_3(a\ot a)=\gamma b$, $f_3(b\ot b)=\lambda b$ and $f_3(c\ot c)=0$.
		
		\item There exists $\lambda\in\K^*$ such that $f_3(a\ot b)=\lambda a$, $f_3(a\ot c)=0$, $f_3(b\ot c)=0$, $f_3(a\ot a)=0$, $f_3(b\ot b)=\lambda b$ and $f_3(c\ot c)=0$.
		
		\item There exists $(\lambda,\gamma,\mu)\in(\K^*)^3$ such that $f_3(a\ot b)=0$, $f_3(a\ot c)=0$, $f_3(b\ot c)=0$, $f_3(a\ot a)=\lambda a$, $f_3(b\ot b)=\gamma b$ and $f_3(c\ot c)=\mu c$.
		
		\item There exists $(\lambda,\gamma)\in(\K^*)^2$ such that $f_3(a\ot b)=0$, $f_3(a\ot c)=0$, $f_3(b\ot c)=0$, $f_3(a\ot a)=\lambda c$, $f_3(b\ot b)=\gamma c$ and $f_3(c\ot c)=0$.
		
		\item There exists $(\lambda,\gamma)\in(\K^*)^2$ such that $f_3(a\ot b)=0$, $f_3(a\ot c)=0$, $f_3(b\ot c)=0$, $f_3(a\ot a)=\lambda c$, $f_3(b\ot b)=\gamma b$ and $f_3(c\ot c)=0$.
		
		\item \label{deuxlettresun} There exists $(\lambda,\gamma)\in(\K^*)^2$ such that $f_3(a\ot b)=0$, $f_3(a\ot c)=0$, $f_3(b\ot c)=0$, $f_3(a\ot a)=\lambda a$, $f_3(b\ot b)=\gamma b$ and $f_3(c\ot c)=0$.
		
		\item \label{deuxlettresdeux}  There exists $\lambda\in\K^*$ such that $f_3(a\ot b)=0$, $f_3(a\ot c)=0$, $f_3(b\ot c)=0$, $f_3(a\ot a)=\lambda b$, $f_3(b\ot b)=0$ and $f_3(c\ot c)=0$.
		
		\item\label{deuxlettrestrois} There exists $\lambda\in\K^*$ such that $f_3(a\ot b)=0$, $f_3(a\ot c)=0$, $f_3(b\ot c)=0$, $f_3(a\ot a)=\lambda a$, $f_3(b\ot b)=0$ and $f_3(c\ot c)=0$.
			
		\item The map $f_3$ is the null map.
	\end{enumerate}
\end{lemma}

\begin{proof}
We use the fact that the map $f_3$ is associative and commutative, and then, we get the lemma by direct quite long calculations.
\end{proof}

\begin{prop}
		Let $X=\{a,b,c\}$ be an alphabet of cardinality 3 and let $\Box$ be a weak stuffle product. In the previous lemma, if $f_3$ satisfies one of the items \ref{un}, \ref{deux}, \ref{cinq}, \ref{six}, \ref{huit}, \ref{quatorze}, \ref{quinze} then either ($f_1\equiv 0$ and $f_2\equiv 0$) or ($f_1(a\otimes b )=1$ and  $f_2(a\otimes b)=1$ for $(a,b)\in X^2$).
\end{prop}

\section{Weak stuffle product and Hopf algebras}
If $\Box$ is the classical shuffle product or the classical stuffle product then the algebra  $(\K\langle X\rangle,\Box)$ can be equipped with a compatible coalgebra structure, thanks to the deconcatenation coproduct, which makes it into a Hopf algebra. Are there other weak stuffle products compatible with the deconcatenation? We begin by recalling the Hopf algebra construction for stuffle algebras given in \cite{Hoffman2000,Hoffman2017,Hoffman2018}. We then turn to the case of weak stuffle algebras.

\begin{theo}\label{theoHopfalgebra}
	Let $X$ be a countable alphabet, let $\K\langle X\rangle$ be the vector space generated by words on the alphabet $X$. We assume there exists at least one product $\diamond$ on $\K.X$ which is commutative and associative. We define the product $\star$ and the coproduct of deconcatenation $\Delta$ by:
	\[au\star bv=a(u\star bv)+b(au\star v)+(a\diamond b)(u\star v)\]
	and \[\Delta(w)=\sum_{\substack{(u,v)\in (\K\langle X\rangle)^2,\\ uv=w}}u\ot v\]
	for any letters $a$ and $b$ and any words $u$, $v$ and $w$.
	
	Then $(\K\langle X\rangle,\star,\Delta)$ is a Hopf algebra.
\end{theo}

\begin{proof}
	This theorem is proven in \cite{Hoffman2000,Hoffman2017,Hoffman2018} by induction and using the filtration given by the length of words.
\end{proof}

\begin{theo}
	Let $X$ be a countable alphabet of cardinality $n\in\N\cup\{+\infty\}$ and let $\Box$ be a weak stuffle product on $\K\langle X\rangle$. We denote by $\Delta$ the deconcatenation coproduct. If $\Delta$ respects $\Box$ (\emph{i.e.} if $\Delta$ is an algebra morphism) then the underlying weak shuffle product is the classical shuffle product. 
\end{theo}

\begin{proof}
	Let $\Box$ be a weak stuffle product. We assume the deconcatenation respects $\Box$. Then, for any distinct letters $a$ and $b$:
	\begin{align*}
	\Delta(a\Box a)=&\left(f_1(a\ot a)+f_2(a\ot a)\right)\Delta(aa)+\Delta(f_3(a\ot a))\\
	=&\left(f_1(a\ot a)+f_2(a\ot a)\right)\Delta(aa)+k(a\ot a)\Delta(g(a\ot a))\\
	=&\left(f_1(a\ot a)+f_2(a\ot a)\right)(aa\ot 1 + a\ot a+1\ot aa)\\
	&+k(a\ot a)\left(g(a\ot a)\ot 1+1\ot g(a\ot a)\right)\\
	=&\Delta(a)\Box\Delta(a)\\
	=&\left(f_1(a\ot a)+f_2(a\ot a)\right)(aa\ot 1+1\ot aa)+2a\ot a\\
	&+k(a\ot a)\left(g(a\ot a)\ot 1+1\ot g(a\ot a)\right),\\
	\Delta(a\Box b)=&f_1(a\ot b)\Delta(ab)+f_1(b\ot a)\Delta(ba)+ k(a\ot b)\Delta(g(a\ot b))\\
	=&f_1(a\ot b)(ab\ot 1+a\ot b +1\ot ab)+f_1(b\ot a)(ba\ot 1+b\ot a +1\ot ba)\\
	&+k(a\ot b)\left(g(a\ot b)\ot 1+1\ot g(a\ot b)\right)\\
	=&\Delta(a)\Box\Delta(b)=f_1(a\ot b)(ab\ot 1+1\ot ab)+f_1(b\ot a)(ba\ot 1+1\ot ba)+a\ot b+b\ot a\\
	&+k(a\ot b)\left(g(a\ot b)\ot 1+1\ot g(a\ot b)\right).
	\end{align*}
	So, $f_1(a\ot a)=f_2(a\ot a)=f_1(a\ot b)=f_1(b\ot a)=1$.
	
	The reversal is a particular case of Theorem \ref{theoHopfalgebra}.
\end{proof}

\section{Computation programs} \label{computation}
We give computation programs realised to compute the weak shuffle of two words or to prove Lemma \ref{lemmaprogramme}. In the sequel we assume the alphabet $X$ is the set of integers $\{1,\dots, c\}$ and a word is a list $[i_1,\dots,i_n]$.

We first present a function which computes the weak shuffle product of two words. This function, called {\ttfamily weak\_shuffle\_product}, takes as entries a list  Rules which coresponds to the values taken by $f_1$ and $f_2$ and two lists  w1 and w2 which represent the two words to use for computations. We assume 
\begin{align*}
\text{\ttfamily Rules}=&\Bigg[f_1(1\ot 2),\dots,f_1(1\ot c),\dots f_1(c\ot 1),\dots, f_1(c\ot c-1),\\
&f_1(1\ot 1), f_2(1\ot 1),\dots, f_1(c\ot c),f_2(c\ot c)\Bigg].
\end{align*} 
As exit, the function return a list. Each element of the result is a list of two elements A and B: A is the number of times the word represented by B appears in the weak shuffle product of w1 and  w2.
               
\begin{lstlisting}[language=MuPAD]
weak_shuffle_product(Rules,w1,w2):=block([n1,n2,u1,u2,temp,res,i,j,
                                v1a,v1b,v2a,v2b,P1,P2,g,d,L,r,s,c],
  /*-------- Initialisation of the values of the left side and 
                                          the right side --------*/    
  g:0,
  d:0,

  /*------ Computation of the cardinality of the alphabet. ------*/    
  r:length(Rules),
  s:sort(solve(c*(c+1)=r)),
  c:subst(s[2],c),

  /*------ Message if the variable Rules does not correspond 
                                           to an alphabet. ------*/    
  if (notequal(c,floor(c)) or c<1) then print("erreur"),

  /*------ Computation of the length of words w1 and w2. ------*/
  n1:length(w1),
  n2:length(w2),
  
  /*-------- We use the commutativity of the  weak shuffle product 
         to avoid some sub-cases. The word with the smallest length 
                                         is on the left. --------*/
  if n1<=n2 then (
  	u1:[[1],w1],
	u2:[[1],w2]
  )
  else (  u1:[[1],w2],
	u2:[[1],w1],
	temp:n1,
	n1:n2,
	n2:temp
	),    
  
  res:[[0],[]],

  /*-------- We will use a recursive call. --------*/    
  if equal(n1,0) then (
  	/*---- Limit case: w1 is the empty word and 
  	                                        w2 is any word. ----*/
	res:[[[1],u2[2]]]
  )    
  else (  
    /*---- We compute the weak shuffle product thanks to the rela-
     tion: au(wsp)bv=f1(a\ot b)a(u(wsp)vb)+ f2(a\ot b)b(ua(wsp)v)
           here u and v are words and a and b are letters. ----*/
	v1a:create_list(u1[2][i],i,2,n1),
	v1b:u1[2][1],
	v2a:create_list(u2[2][i],i,2,n2),
	v2b:u2[2][1],
	P1:[],
	P2:[],
	
	/*--- We detemine f_1(v1b\ot v2b) and f_2(v1b\ot v2b). ---*/
    if equal(v1b,v2b) then (
	  g:Rules[r+2*(-c+v1b)-1], 
	  d:Rules[r+2*(-c+v1b)]
	),
	if (v1b<v2b) then (
	  g:Rules[(v1b-1)*(c-1)+v2b-1],
	  d:Rules[(v2b-1)*(c-1)+v1b]
	),
	if (v1b>v2b) then (
	  g:Rules[(v1b-1)*(c-1)+v2b],
	  d:Rules[(v2b-1)*(c-1)+v1b-1]
	),

	/*-------- Recursive call. --------*/
	if g>0 then (
	  P1:weak_shuffle_product(Rules,v1a,u2[2]),
	  P1:create_list([g*P1[i][1],append([v1b],P1[i][2])],
	                                              i,1,length(P1))
	),
	if d>0 then (
	  P2:weak_shuffle_product(Rules,u1[2],v2a),    
	  P2:create_list([d*P2[i][1],append([v2b],P2[i][2])],
	                                              i,1,length(P2))
	),
	res:append(P1,P2)       
  ),
	
  /*------ We rewrite the result for having only one occurence of 
                                   each distinct words. --------*/
  L:create_list(res[i][2],i,1,length(res)),
  L:unique(L),
  res:create_list([ratsimp(sum(if equal(L[i],res[j][2]) then res[j][1] 
                      else 0, j,1, length(res))),L[i]],i,1,length(L)),

  return(res)
);
\end{lstlisting}

In the sequel, the functions aim at proving if the following statement is true or not for some low $n$.
Let $n$ be a positive integer and let $w_1$, $w_2$ and $w$ be three non-empty words of length $n$ such that $w_1\leq w_2\leq w$ and $w_1<w$. Then $\max(w_1\underset{9}{\Box} w_2)<\max(w\underset{9}{\Box} w)$? It is trivial for $n=1$. For $n=2$, it comes from computations doing in the proof of \ref{Cneuf}. Thus, those cases are not treated.

The function {\ttfamily words} aims at building all words of length $n$ with an alphabet of cardinality $c$. It takes as entries the integers $n$ and $c$ and returns a list where each element is a list coresponding to a word. In the result, words are ordered by the ascending order.

\begin{lstlisting}[language=MuPaD]
words(n,c):=block([res,i,j,U],
  res:[],
  if n=1 then res:create_list([i],i,1,c),
  if n>1 then (
    U:words(n-1,c),
    res:create_list(append(U[i],[j]),j,1,c,i,1,length(U))
  ),
  return(sort(res))
);
\end{lstlisting}
 
The function {\ttfamily spectrum\_product} aims at determining words appearing in the weak shuffle product of two words w1 and w2. It takes as entries a list Rules which gives the rules of computation for the weak shuffle product, an integer  $r$ which is the length of the list  Rules, an integer $c$ which is the cardinality of the alphabet, and two lists w1 and  w2 which represent the two words to use for computations.

As exit, the function return a list ordered thanks to the ascending order where each element is a list representing a word appearing in the weak shuffle product of two words  w1 and  w2.

\begin{lstlisting}[language=MuPAD]
spectrum_product(Rules,r,c,w1,w2):=block([n1,n2,u1,u2,temp,res,i,j,
                                        v1a,v1b,v2a,v2b,P1,P2,g,d],
  /*-------- Initialisation of the values of 
                        the left side and the right side --------*/    
  g:0,
  d:0,
  /*------- Computation of the length of words w1 and w2. -------*/
  n1:length(w1),
  n2:length(w2),
 /*-------- We use the commutativity of the  weak shuffle product 
         to avoid some sub-cases. The word with the smallest length 
                                         is on the left. --------*/ 
  if n1<=n2 then (
    u1:w1,
    u2:w2
  )
  else (  u1:w2,
    u2:w1,
    temp:n1,
    n1:n2,
    n2:temp
  ),    
  res:[],

  /*-------- We will use a recursive call. --------*/    
  if equal(n1,0) then (
    /*--- Limit case: w1 is the empty word and w2 is any word. ---*/
    res:[u2]
  )    
  else (  
    /*---- We compute the weak shuffle product thanks to the rela-
    tion: au(wsp)bv=f1(a\ot b)a(u(wsp)vb)+ f2(a\ot b)b(ua(wsp)v)
    here u and v are words and a and b are letters. ----*/
    v1a:deleten(u1,1),
    v1b:u1[1],
    v2a:deleten(u2,1),
    v2b:u2[1],
    P1:[],
    P2:[],
    
    /*----- We detemine f_1(v1b\ot v2b) and f_2(v1b\ot v2b). -----*/
    if equal(v1b,v2b) then (
      g:Rules[r+2*(-c+v1b)-1], 
      d:Rules[r+2*(-c+v1b)]
    ),
	
    if (v1b<v2b) then (
      g:Rules[(v1b-1)*(c-1)+v2b-1],
      d:Rules[(v2b-1).(c-1)+v1b]
    ),
    if (v1b>v2b) then (
      g:Rules[(v1b-1)*(c-1)+v2b],
      d:Rules[(v2b-1).(c-1)+v1b-1]
    ),
	
    /*-------- Recursive call. --------*/
    if g>0 then (
      P1:spectrum_product(Rules,r,c,v1a,u2),
      P1:create_list(append([v1b],P1[i]),i,1,length(P1))
    ),
    if d>0 then (
      P2:spectrum_product(Rules,r,c,u1,v2a),
      P2:create_list(append([v2b],P2[i]),i,1,length(P2))
    ),
	
    res:append(P1,P2)
  ),

  /*------ Words are written once with the ascending order. ------*/
  res:sort(unique(res)),
  return(res)
);
\end{lstlisting}

The function {\ttfamily maximum\_product}  takes as entries a list Rules corresponding to the weak shuffle product, an integer $r$ which is the length of Rules, an integer  $c$ which is the cardinality of the alphabet, an integer $n$ which is the length of words used, a list $W$ which represents the list of words of length $n$, an integer $l$ which is the length of $W$, an integer $k$ which is the level of computation. The function returns a list of length $k-5$. The first one is a list of only one element which is $\max(W[6]\underset{9}{\Box}W[6])$. In the result, the element $p$ with $2\leq p\leq k-5$ is a list of two elements $A_p$ and $B_p$ where $A_p=\max(\max(w_1\underset{9}{\Box}w_2))$ with $w_1<W[p]$
and $w_2\leq W[k]$ and $B_p=\max(W[p]\underset{9}{\Box}W[p])$. This function really depends on the weak shuffle product $\underset{9}{\Box}$. 
\begin{lstlisting}[language=MuPaD]
maximum_product(Rules,r,c,n,W,l,k):=block([res,i,P,init],
  res:[],
  if n>1 then (
    /*-------------- W[1]=[1,...,1], W[2]=[1,...,1,2], 
               W[3]=[1,...,1,2,1], W[4]=[1,...,1,2,2], 
           W[5]=[1,...,1,2,1,1], W[6]=[1,...,1,2,1,2], 
             it is enouth to do an initialisation with 
                                W[6]. --------------*/
    if k=6 then (
      init:last(spectrum_product(Rules,r,c,W[6],W[6])), 
      res:[[init]]
    ),
    if (k>6 and k<l+1) then (
      /*----- Recursive call. ---------*/
      res:maximum_product(Rules,r,c,n,W,l,k-1),
      
      /*---- Maximum word in res. ----*/
      P:[last(sort(res[length(res)]))],
      
      /*--- P is filled in maximum words in W[i](wsp)W[k]
      for i:1 thru k-1 do (
        P:append(P,[last(spectrum_product(Rules,r,c,W[i],W[k]))])
      ),
      
      /*--- res is filled in a list of two elements: 
         the maximum in P and the maximum in W[K](spw)W[k]. ---*/ 
      res:append(res,[[last(sort(P)),
                   last(spectrum_product(Rules,r,c,W[k],W[k]))]])
    )
  ),
  return(res)
);
\end{lstlisting}

The function {\ttfamily proof\_statement} determines if the statement given at the beginning of the section is proved for words of length $n$. As entries, it takes a list Rules corresponding to the weak shuffle product and an integer coresponding to the length of words used. It returns a boolean. The boolean is true if the statement if satisfied and false if the statement is not satisfied. Since this function uses {\ttfamily maximum\_product}, it depends on the weak shuffle product $\underset{9}{\Box}$.
\begin{lstlisting}[language=MuPaD]
proof_statement(Rules,n):=block([res,P,U,i,p,c,r,s,W,l],
  /*------- Computation of the cardinality of the alphabet. -------*/    
  r:length(Rules),
  s:sort(solve(c*(c+1)=r)),
  c:subst(s[2],c),
  
  /*-------- Message if the variable Rules
                        does not correspond to an alphabet. --------*/    
  if (notequal(c,floor(c)) or c<1) then print("erreur")
  else(
    /*-------- Computations. --------*/ 
    res:true,
    /*------ Building of words of length n. ------*/
    W:words(n,c),
    l:length(W),
    /*------ Building max(w(wsp)w) and max(max(w_1(wsp)w_2) 
                                       with w_1<w and w_2<=w. ------*/
    P:maximum_product(Rules,r,c,n,W,l,l),
    p:length(P),
    i:2,
    /*------ Checking of the statement at level i. ------*/
    while ( equal(res,true) and i<p+1) do (
      if equal(P[i][1],P[i][2]) then ( res:false),
      i:i+1
    )
  ),
  return(res)
);
\end{lstlisting}

\bibliographystyle{siam}
\bibliography{biblio_stuffle_mammez}

\begin{thebibliography}{10}

\bibitem{Aguiar2004}
{\sc M.~Aguiar and J.-L. Loday}, {\em Quadri-algebras}, J. Pure Appl. Algebra,
  191 (2004), pp.~205--221.

\bibitem{Bradley2005}
{\sc D.~M. Bradley}, {\em Multiple q-zeta values}, Journal of Algebra, 283
  (2005), pp.~752 -- 798.

\bibitem{Chapoton2002}
{\sc F.~Chapoton}, {\em Un théorème de {C}artier-{M}ilnor-{M}oore-{Q}uillen
  pour les bigèbres dendriformes et les algèbres braces}, J. Pure Appl.
  Algebra, 168 (2002), pp.~1--18.

\bibitem{Duchamp2011}
{\sc G.~Duchamp, F.~Hivert, J.-C. Novelli, and J.-Y. Thibon}, {\em
  Noncommutative symmetric functions {VII}: {F}ree quasi-symmetric functions
  revisited}, Ann. Comb., 15 (2011), pp.~655--673.

\bibitem{Duchamp2002}
{\sc G.~Duchamp, F.~Hivert, and J.-Y. Thibon}, {\em Noncommutative symmetric
  functions {VI}: {F}ree quasi-symmetric functions and related algebras},
  Internat. J. Algebra Comput., 12 (2002), pp.~671--717.

\bibitem{Ebrahimi-Fard2006}
{\sc K.~Ebrahimi-Fard and L.~Guo}, {\em Mixable shuffles, quasi-shuffles and
  {H}opf algebras}, Journal of Algebraic Combinatorics, 24 (2006), pp.~83--101.

\bibitem{Ebrahimi-Fard2016a}
{\sc K.~Ebrahimi-Fard, D.~Manchon, and J.~Singer}, {\em Duality and
  $q$-multiple zeta values}, Advances in Mathematics, 298 (2016), pp.~254 --
  285.

\bibitem{Ebrahimi-Fard2016}
{\sc K.~Ebrahimi-Fard, D.~Manchon, and J.~Singer}, {\em {The Hopf Algebra of
  $q$-{M}ultiple {P}olylogarithms with {N}on-positive Arguments}},
  International Mathematics Research Notices, 2017 (2016), pp.~4882--4922.

\bibitem{Foissy2002b}
{\sc L.~Foissy}, {\em Les algèbres de {H}opf des arbres enracinés décorés.
  {II}}, Bull. Sci. Math., 126 (2002), pp.~249--288.

\bibitem{Foissy2007}
\leavevmode\vrule height 2pt depth -1.6pt width 23pt, {\em Bidendriform
  bialgebras, trees and free quasi-symmetric functions}, J. Pure Appl. Algebra,
  209 (2007), pp.~439--459.

\bibitem{Foissy2015}
\leavevmode\vrule height 2pt depth -1.6pt width 23pt, {\em Free quadri-algebras
  and dual quadri-algebras}, ArXiv e-prints,  (2015).

\bibitem{Foissy2016}
{\sc L.~Foissy, F.~Patras, and J.-Y. Thibon}, {\em Deformations of shuffles and
  quasi-shuffles}, Annales de l'Institut Fourier, 66 (2016), pp.~209--237.

\bibitem{Gessel1984}
{\sc I.~M. Gessel}, {\em Multipartite $p$-partitions and inner products of skew
  {S}chur functions}, in Combinatorics and algebra (Boulder, Colo., 1983),
  vol.~34 of Contemp. Math., Amer. Math. Soc., Providence, RI, 1984,
  pp.~289--317.

\bibitem{Guo2000}
{\sc L.~Guo and W.~Keigher}, {\em Baxter {A}lgebras and {S}huffle {P}roducts},
  Advances in Mathematics, 150 (2000), pp.~117 -- 149.

\bibitem{Hoffman1997}
{\sc M.~E. Hoffman}, {\em The algebra of multiple harmonic series}, J. Algebra,
  194 (1997), pp.~477--495.

\bibitem{Hoffman2000}
\leavevmode\vrule height 2pt depth -1.6pt width 23pt, {\em Quasi-shuffle
  products}, J. Algebraic Combin., 11 (2000), pp.~49--68.

\bibitem{Hoffman2018}
\leavevmode\vrule height 2pt depth -1.6pt width 23pt, {\em {Quasi-shuffle
  algebras and applications}}, arXiv e-prints,  (2018), p.~arXiv:1805.12464.

\bibitem{Hoffman2017}
{\sc M.~E. Hoffman and K.~Ihara}, {\em Quasi-shuffle products revisited}, J.
  Algebra, 481 (2017), pp.~293--326.

\bibitem{Hoffman2003a}
{\sc M.~E. Hoffman and Y.~Ohno}, {\em Relations of multiple zeta values and
  their algebraic expression}, Journal of Algebra, 262 (2003), pp.~332 -- 347.

\bibitem{Jian2017}
{\sc R.-Q. Jian}, {\em Quantum quasi-shuffle algebras {II}}, Journal of
  Algebra, 472 (2017), pp.~480 -- 506.

\bibitem{Jian2010}
{\sc R.-Q. {Jian}, M.~{Rosso}, and J.~{Zhang}}, {\em {Quantum Quasi-Shuffle
  Algebras}}, Letters in Mathematical Physics, 92 (2010), pp.~1--16.

\bibitem{Loday2001}
{\sc J.-L. Loday}, {\em Dialgebras}, in Dialgebras and related operads,
  vol.~1763 of Lecture Notes in Math., Springer, Berlin, 2001, pp.~7--66.

\bibitem{Loday1998}
{\sc J.-L. Loday and M.~Ronco}, {\em Hopf algebra of the planar binary trees},
  Adv. Math., 139 (1998), pp.~293--309.

\bibitem{Malvenuto1994}
{\sc C.~Malvenuto}, {\em Produits et coproduits des fonctions quasi-symétriques
  et de l'algèbre des descentes}, PhD thesis, Université du Québec à Montréal,
  Laboratoire de Combinatoire et d'Informatique Mathématique, 1994.

\bibitem{Malvenuto1995}
{\sc C.~Malvenuto and C.~Reutenauer}, {\em Duality between quasi-symmetric
  functions and the {S}olomon descent algebra}, J. Algebra, 177 (1995),
  pp.~967--982.

\bibitem{Mammez}
{\sc C.~Mammez}, {\em A propos de l'algèbre de {H}opf des mots tassés {WM}at},
  Bulletin des Sciences Mathématiques,  (2018).

\bibitem{Ronco2001}
{\sc M.~Ronco}, {\em A {M}ilnor-{M}oore theorem for dendriform {H}opf
  algebras}, C. R. Acad. Sci. Paris S\'er. I Math., 332 (2001), pp.~109--114.

\bibitem{Schlesinger2001}
{\sc K.-G. Schlesinger}, {\em {S}ome remarks on $q$-deformed multiple
  polylogarithms}, arXiv Mathematics e-prints,  (2001), p.~math/0111022.

\bibitem{Singer2015}
{\sc J.~Singer}, {\em On {$q$}-analogues of multiple zeta values}, Funct.
  Approx. Comment. Math., 53 (2015), pp.~135--165.

\bibitem{Singer2016}
{\sc J.~Singer}, {\em On {B}radley's q-{MZV}s and a generalized {E}uler
  decomposition formula}, Journal of Algebra, 454 (2016), pp.~92 -- 122.

\bibitem{Singer2016a}
{\sc J.~Singer}, {\em $q$-{A}nalogues of multiple zeta values and their
  application in renormalization}, PhD thesis, {D}er {N}aturwissenschaftlichen
  {F}akult{ä}t, der {F}riedrich-{A}lexander-{U}nivedrsit{ä}t, 2016.

\bibitem{Vallette2008}
{\sc B.~Vallette}, {\em Manin products, {K}oszul duality, {L}oday algebras and
  {D}eligne conjecture}, J. Reine Angew. Math., 620 (2008), pp.~105--164.

\bibitem{Vargas2014}
{\sc Y.~Vargas}, {\em Hopf algebra of permutation pattern functions}, in {26th
  International Conference on Formal Power Series and Algebraic Combinatorics
  (FPSAC 2014)}, L.~J. Billera and I.~Novik, eds., vol.~AT of DMTCS
  Proceedings, Chicago, United States, 2014, Discrete Mathematics and
  Theoretical Computer Science, pp.~839--850.

\bibitem{Zagier1994}
{\sc D.~Zagier}, {\em Values of zeta functions and their applications}, in
  First {E}uropean {C}ongress of {M}athematics, {V}ol. {II} ({P}aris, 1992),
  vol.~120 of Progr. Math., Birkh\"{a}user, Basel, 1994, pp.~497--512.

\bibitem{Zhao2007}
{\sc J.~Zhao}, {\em {Multiple $q$-zeta functions and multiple
  $q$-polylogarithms}}, The Ramanujan Journal, 14 (2007), pp.~189--221.

\bibitem{Zudilin2003}
{\sc W.~Zudilin}, {\em {Algebraic relations for multiple zeta values}}, Russian
  Mathematical Surveys, 58 (2003), pp.~1--29.

\end{thebibliography}
\end{document}